\documentclass[11pt,
]{amsart}
\usepackage{
amssymb, graphicx
}
\usepackage{mathrsfs} 
\usepackage{amsthm}
\usepackage{graphicx,color}

\newtheorem{theorem}{Theorem}
\newtheorem{lemma}{Lemma}
\newtheorem{proposition}{Proposition}

\newtheorem{remark}{Remark}
\theoremstyle{remark}

\date{\today}

\title[On an inverse boundary value problem for a nonlinear elastic wave equation]{On an inverse boundary value problem for a nonlinear elastic wave equation}
  \author[G. Uhlmann]{Gunther Uhlmann}
\address{Department of Mathematics, University of Washington, Seattle, WA 98195, USA; Institute for Advanced Study, 
The Hong Kong University of Science and Technology, Kowloon, Hong Kong, China (\tt{gunther@math.washington.edu})
}
  
  \author[J. Zhai]{Jian Zhai}
\address{Institute for Advanced Study,
  The Hong Kong University of Science and Technology, Kowloon, Hong Kong, China
  (\tt{jian.zhai@outlook.com}).}
 
\begin{document}
\begin{abstract}
We consider an inverse boundary value problem for a nonlinear elastic wave equation which was studied in \cite{de2018nonlinear}. We show that all the parameters appearing in the equation can be uniquely determined from boundary measurements under certain geometric assumptions. The proof is based on second order linearization and Gaussian beams.
\end{abstract}
\maketitle

\section{Introduction}

Consider the initial boundary value problem for the quasilinear elastic wave equation
\begin{equation}\label{elastic_eq}
\begin{split}
&\rho\frac{\partial^2u}{\partial t^2}-\nabla\cdot S(x,u)=0,\quad (t,x)\in (0,T)\times\Omega,\\
&u(t,x)=f(t,x),\quad (t,x)\in (0,T)\times\partial \Omega,\\
&u(0,x)=\frac{\partial}{\partial t}u(0,x)=0,\quad x\in \Omega.
\end{split}
\end{equation}
 Here $\Omega\subset\mathbb{R}^3$ is a bounded domain with smooth boundary $\partial\Omega$. We denote $x=(x_1,x_2,x_3)$ to be the Cartesian coordinates. Then we can write the displacement vector as $u=(u_1,u_2,u_3)$ under the Cartesian coordinates. 
The stress tensor $S$ has the form
\begin{equation}
\begin{split}
S_{ij}=&\lambda \varepsilon_{mm}\delta_{ij}+\lambda\widetilde{\varepsilon}_{mm}\frac{\partial u_i}{\partial x_j}+2\mu\left(\varepsilon_{ij}+\widetilde{\varepsilon}_{jn}\frac{\partial u_i}{\partial x_n}\right)\\
&+ \mathscr{A}\widetilde{\varepsilon}_{in}\widetilde{\varepsilon}_{jn}+\mathscr{B}(2\widetilde{\varepsilon}_{nn}\widetilde{\varepsilon}_{ij}+\widetilde{\varepsilon}_{mn}\widetilde{\varepsilon}_{mn}\delta_{ij})+\mathscr{C}\widetilde{\varepsilon}_{mm}\widetilde{\varepsilon}_{nn}\delta_{ij}+\mathcal{O}(u^3),
\end{split}
\end{equation}
where $\varepsilon$ is the strain tensor defined as
\[
\varepsilon_{ij}(u)=\frac{1}{2}\left(\frac{\partial u_i}{\partial x_j}+\frac{\partial u_j}{\partial x_i}+\frac{\partial u_k}{\partial x_i}\frac{\partial u_k}{\partial x_j}\right),
\]
and $\widetilde{\varepsilon}$ is the linearized strain tensor
\[
\widetilde{\varepsilon}_{ij}(u)=\frac{1}{2}\left(\frac{\partial u_i}{\partial x_j}+\frac{\partial u_j}{\partial x_i}\right).
\]
By using the notation $\mathcal{O}(u^3)$ we are considering the small displacement asymptotics. The functions 
\[
\lambda(x),\mu(x),\rho(x),\mathscr{A}(x),\mathscr{B}(x),\mathscr{C}(x)
\]
 are all smooth on $\overline{\Omega}$. The parameters $\lambda$ and $\mu$ are called Lam\'e moduli and $\rho$ is the density. This model is widely used and can be found in \cite{landau1960theory,de2003finite,de2018nonlinear}.\\

In this article, we study the inverse problem of recovering the elastic parameters $\lambda,\,\mu,\,\rho,\,\mathscr{A},\,\mathscr{B},\,\mathscr{C}$ from \textit{displacement-to-traction map}
\[
\Lambda:f\rightarrow \nu\cdot S(x,u)\vert_{(0,T)\times\partial\Omega},
\]
where $\nu$ is the exterior normal unit vector to $\partial \Omega$.

The well-definedness of $\Lambda$ for small $f$ is guaranteed by the well-posedness of \eqref{elastic_eq} with small boundary data:
\begin{proposition}[\text{\cite[Theorem 2]{de2018nonlinear}}] 
Assume $f \in C^m([0,T]\times\partial\Omega), m\geq 3$ is supported away from $t=0$. Then there exists $\epsilon_0>0$ such that for $\|f\|_{C^m}<\epsilon_0$ there exists a unique solution
\[
u\in \bigcap_{k=0}^m C^k([0,T];\, W^{m-k,2})([0,T]\times\Omega)).
\]
\end{proposition}

Denote $S=S^L+S^N$, where $S^L$ is the linearized stress
\[
S^L_{ij}(x,u)=\lambda \widetilde{\varepsilon}_{mm}\delta_{ij}+2\mu \widetilde{\varepsilon}_{ij},
\]
and
\[
\begin{split}
S^N_{ij}(x,u)=&\frac{\lambda+\mathscr{B}}{2}\frac{\partial u_m}{\partial x_n}\frac{\partial u_m}{\partial x_n}\delta_{ij}+\mathscr{C}\frac{\partial u_m}{\partial x_m}\frac{\partial u_n}{\partial x_n}\delta_{ij}+\frac{\mathscr{B}}{2}\frac{\partial u_m}{\partial x_n}\frac{\partial u_n}{\partial x_m}\delta_{ij}+\mathscr{B}\frac{\partial u_m}{\partial x_m}\frac{\partial u_j}{\partial x_i}\\
&+\frac{\mathscr{A}}{4}\frac{\partial u_j}{\partial x_m}\frac{\partial u_m}{\partial x_i}+(\lambda+\mathscr{B})\frac{\partial u_m}{\partial x_m}\frac{\partial u_i}{\partial x_j}\\
&+\left(\mu+\frac{\mathscr{A}}{4}\right)\left(\frac{\partial u_m}{\partial x_i}\frac{\partial u_m}{\partial x_j}+\frac{\partial u_i}{\partial x_m}\frac{\partial u_j}{\partial x_m}+\frac{\partial u_i}{\partial x_m}\frac{\partial u_m}{\partial x_j}\right)+\mathcal{O}(u^3).
\end{split}
\]

The linear elastic wave equation reads
\begin{equation}\label{elastic_linear_eq}
\begin{split}
&\rho\frac{\partial^2u}{\partial t^2}-\nabla\cdot S^L(x,u)=0,\quad (t,x)\in (0,T)\times\Omega,\\
&u(t,x)=f(t,x),\quad (t,x)\in (0,T)\times\partial \Omega,\\
&u(0,x)=\frac{\partial}{\partial t}u(0,x)=0,\quad x\in \Omega.
\end{split}
\end{equation}
We denote the Dirichlet-to-Neumann map for the above linear elastic wave equation as
\[
\Lambda^{lin}:f\rightarrow \nu\cdot S^L(x,u)\vert_{(0,T)\times\partial\Omega}.
\]

The \textit{S}- and \textit{P}- wavespeeds are related with the Lam\'e moduli $\lambda,\,\mu$ and the density $\rho$ in the following way
\[
c_S=\sqrt{\frac{\mu}{\rho}},\quad c_P=\sqrt{\frac{\lambda+2\mu}{\rho}}.
\]
We will assume
\[
\mu>0,\quad 3\lambda+2\mu>0~\text{on}~\overline{\Omega}.
\]
Then $c_P>c_S$ in $\overline{\Omega}$. Denote the Riemannian metrics associated with $P/S$- wave speeds to be
\[
g_{P/S}=c_{P/S}^{-2}\mathrm{d}s^2,
\]
where $\mathrm{d}s^2$ is the Euclidean metric.
Then \textit{P}- and \textit{S}- waves travel along geodesics in the Riemannian manifolds $(\Omega,g_P)$ and $(\Omega,g_S)$ respectively. Let $\mathrm{diam}_{P/S}(\Omega)$ be the diameter of $\Omega$ with respect to $g_{P/S}$. More precisely, 
\[
\mathrm{diam}_{P/S}(\Omega)=\sup\{\text{lengths of all geodesics in }(\Omega,g_{P/S})\}
\]

For the inverse problem, one can first recover the Dirchlet-to-Neumann map $\Lambda^{lin}$ for the linear elastic wave equation \eqref{elastic_linear_eq} by first order linearization of $\Lambda$ (cf. \cite{de2018nonlinear})
\[
\frac{\partial}{\partial\epsilon}\Lambda(\epsilon f)\vert_{\epsilon=0}=\Lambda^{lin}(f).
\]
It was shown in \cite{hansen2003propagation} that from  $\Lambda^{lin}$ one can recover the scattering relation associated to the wave speeds $c_P$ and $c_S$. Using the result of \cite{stefanov2017local} one can determine $c_S$ and $c_P$ if the foliation condition is satisfied for both metrics $g_{P/S}$ and $T$ is larger than $\mathrm{diam}_{S}(\Omega)$. For a more precise statement see \cite[Theorem 1.4]{stefanov2016boundary}. Recall that  a Riemannian manifold $(M,g)$ satisfies the foliation condition if it can be foliated by strictly convex hypersurfaces \cite{uhlmann2016inverse}. The foliation condition is satisfied for $(\Omega,g_{P/S})$, for instance, if $\partial\Omega$ is strictly convex (with respect to $g_{P/S}$) and the  wave speeds $c_{P/S}$ increase with depth. They are also satisfied under some additional conditions on the curvature (see \cite{paternain2019geodesic, SUV3}), if $\Omega$ is simply connected. As pointed out in \cite{stefanov2016boundary} the foliation condition is a natural generalization of the Herglotz \cite{herglotz1905elastizitaet} and the Wieckert-Zoeppritz \cite{wiechert1907erdbebenwellen} conditions. The foliation condition allows for conjugate points. If the boundary is strictly convex for $g_{P/S}$  and there are no conjugate points for $g_{P/S}$, the uniqueness of $c_P$ and $c_S$ was shown by Rachele in \cite{rachele2000inverse}.  One can in fact determine the three parameters $\lambda, \mu, \rho$, assuming further $c_P\ne 2c_S$ except at isolated points in $\overline{\Omega}$, under the foliation condition \cite{bhattacharyya2018local} or the no conjugate points condition (an extra curvature condition is needed) \cite{rachele2003uniqueness}. We summarize the results for the linear elastic wave equation \eqref{elastic_linear_eq} in the following.

\begin{proposition}
Assume $T>\max\{\mathrm{diam}_S(\Omega),\mathrm{diam}_P(\Omega)\}$, $\partial\Omega$ is strictly convex with respect to $g_{P/S}$, and either of the following conditions holds
\begin{enumerate}
\item $(\Omega,g_{P/S})$ has no conjugate points;
\item $(\Omega,g_{P/S})$ satisfies the foliation condition.
\end{enumerate}
Then $\Lambda^{lin}$ uniquely determines $\frac{\mu}{\rho}$ and $\frac{\lambda}{\rho}$ in $\overline{\Omega}$. Assume further that $\lambda=2\mu$ only at isolated points in $\overline{\Omega}$. When condition $(1)$ is satisfied, assume also that $(\Omega,g_P)$ is negatively curved. Then $\rho$ is uniquely determined.
\end{proposition}

\begin{remark}
Notice that when $(\Omega,g_{P/S})$ has no conjugate points, an extra curvature condition on $(\Omega,g_P)$ is needed for the unique determination of $\rho$. This is due to the fact that the injectivity (up to natural obstructions) of the related tensor tomography problem is established only under extra curvature conditions. The curvature condition can be relaxed using the results of \cite{sharafutdinov1992integral,dairbekov2006integral,paternain2015invariant}.
\end{remark}


 In this paper we mainly focus on the determination of the nonlinear elastic parameters $\mathscr{A},\mathscr{B},\mathscr{C}$. In \cite{de2018nonlinear}, the authors proved the uniqueness of $\mathscr{A}$ and $\mathscr{B}$, by analyzing the nonlinear interactions of distorted plane waves. The approach originated from \cite{kurylev2018inverse}, and has been successfully used to study inverse problems for nonlinear hyperbolic equations \cite{chen2019detection,kurylev2014inverse,lassas2018inverse,uhlmann2018determination,wang2019inverse}. 
We will present an alternative approach to the proof of the uniqueness of $\mathscr{A}$ and $\mathscr{B}$, and further extend to the uniqueness of $\mathscr{C}$. Our work is still based on the higher order linearization utilized in aforementioned work, but instead of distorted plane waves we will use Gaussian beams. We note here that Gaussian beams have been used to study various inverse problems \cite{bao2014sensitivity, belishev1996boundary, dos2016calderon,
feizmohammadi2019timedependent, feizmohammadi2019inverse,feizmohammadi2019recovery,katchalov1998multidimensional}. We emphasize here that Gaussian beams can be constructed allowing conjugate points.\\

We summarize the main theorem of this article here:
\begin{theorem}
Assume $T>2\,\max\{\mathrm{diam}_S(\Omega),\mathrm{diam}_P(\Omega)\}$, $\partial\Omega$ is strictly convex with respect to $g_{P/S}$, and either of the following conditions holds
\begin{enumerate}
\item $(\Omega,g_{P/S})$ has no conjugate points;
\item $(\Omega,g_{P/S})$ satisfies the foliation condition.
\end{enumerate}
Assume $\lambda,\mu,\rho$ can be recovered from $\Lambda^{lin}$. Then $\Lambda$ determines $\lambda,\mu,\rho,\mathscr{A},\mathscr{B},\mathscr{C}$ in $\overline{\Omega}$ uniquely.
\end{theorem}

The rest of this paper is organized as follows. In Section \ref{linearization}, we carry out the second order linearization of the displacement-to-traction map and derive an integral identity, from which we can recover the parameters of interest. In Section \ref{gaussianbeam}, we construct Gaussian beam solutions to linear elastic wave equation, for both \textit{P}- and \textit{S}- waves. Finally, Section \ref{main} is devoted to the proof of the main theorem.
\section{Second-order linearization of displacement-to-traction map}\label{linearization}
We will apply the higher order linearization technique introduced in \cite{kurylev2018inverse} to the displacement-to-traction map $\Lambda$, and arrive at an integral identity which could be used for the recovery of the parameters. The linearization of $\Lambda$ itself has already been used in \cite{de2018nonlinear}. Higher order linearization of Dirichlet-to-Neumann map and the resulted integral identities for semilinear and quasilinear elliptic equations are used \cite{sun1997inverse,kang2002identification,assylbekov2017direct,carstea2019reconstruction,lassas2019inverse,lassas2019partial,feizmohammadi2019inverse,krupchyk2019partial}.  Assume $u$ solves \eqref{elastic_eq} with Dirichlet boundary value
\[
f=\epsilon_1f^{(1)}+\epsilon_2f^{(2)}.
\]
Denote $u^{(j)},\, j=1,2$ to be the solution to the linearized elastic wave equation with boundary value $f^{(j)}$, i.e.,
\begin{equation}\label{elastic_eq_linearized}
\begin{split}
&\rho\frac{\partial^2u^{(j)}}{\partial t^2}-\nabla\cdot S^L(x,u^{(j)})=0,\quad (t,x)\in (0,T)\times\Omega,\\
&u^{(j)}=f^{(j)},\quad \text{on } (0,T)\times\partial \Omega,\\
&u^{(j)}(0,x)=\frac{\partial}{\partial t}u^{(j)}(0,x)=0,\quad x\in \Omega.
\end{split}
\end{equation}

Applying $\frac{\partial^2}{\partial\epsilon_1\partial\epsilon_2}$ to \eqref{elastic_eq}, we obtain the equation for $\mathcal{U}^{(12)}=\frac{\partial^2}{\partial\epsilon_1\partial\epsilon_2}u\vert_{\epsilon_1=\epsilon_2=0}$.
\begin{equation}\label{elastic_U2}
\begin{split}
&\rho\frac{\partial^2}{\partial t^2}\mathcal{U}^{(12)}-\nabla\cdot S^L(x,\mathcal{U}^{(12)})=\nabla\cdot G(u^{(1)},u^{(2)}),\quad (t,x)\in (0,T)\times\Omega,\\
&\mathcal{U}^{(12)}(t,x)=0,\quad (t,x)\in (0,T)\times\partial \Omega,\\
&\mathcal{U}^{(12)}(0,x)=\frac{\partial}{\partial t}\mathcal{U}^{(12)}(0,x)=0,\quad x\in \Omega.
\end{split}
\end{equation}
Here
\begin{equation}\label{G_form}
\begin{split}
G_{ij}(u^{(1)},u^{(2)})=&(\lambda+\mathscr{B})\frac{\partial u^{(1)}_m}{\partial x_n}\frac{\partial u^{(2)}_m}{\partial x_n}\delta_{ij}+2\mathscr{C}\frac{\partial u^{(1)}_m}{\partial x_m}\frac{\partial u^{(2)}_n}{\partial x_n}\delta_{ij}+\mathscr{B}\frac{\partial u_m^{(1)}}{\partial x_n}\frac{\partial u_n^{(2)}}{\partial x_m}\delta_{ij}\\
&+\mathscr{B}\left(\frac{\partial u^{(1)}_m}{\partial x_m}\frac{\partial u^{(2)}_j}{\partial x_i}+\frac{\partial u^{(2)}_m}{\partial x_m}\frac{\partial u^{(1)}_j}{\partial x_i}\right)+\frac{\mathscr{A}}{4}\left(\frac{\partial u^{(1)}_j}{\partial x_m}\frac{\partial u^{(2)}_m}{\partial x_i}+\frac{\partial u^{(2)}_j}{\partial x_m}\frac{\partial u^{(1)}_m}{\partial x_i}\right)\\
&+(\lambda+\mathscr{B})\left(\frac{\partial u^{(1)}_m}{\partial x_m}\frac{\partial u^{(2)}_i}{\partial x_j}+\frac{\partial u^{(2)}_m}{\partial x_m}\frac{\partial u^{(1)}_i}{\partial x_j}\right)\\
&+\left(\mu+\frac{\mathscr{A}}{4}\right)\Bigg(\frac{\partial u^{(1)}_m}{\partial x_i}\frac{\partial u^{(2)}_m}{\partial x_j}+\frac{\partial u^{(2)}_m}{\partial x_i}\frac{\partial u^{(1)}_m}{\partial x_j}+\frac{\partial u^{(1)}_i}{\partial x_m}\frac{\partial u^{(2)}_j}{\partial x_m}+\frac{\partial u^{(2)}_i}{\partial x_m}\frac{\partial u^{(1)}_j}{\partial x_m}\\
&\quad\quad+\frac{\partial u^{(1)}_i}{\partial x_m}\frac{\partial u^{(2)}_m}{\partial x_j}+\frac{\partial u^{(2)}_i}{\partial x_m}\frac{\partial u^{(1)}_m}{\partial x_j}\Bigg).
\end{split}
\end{equation}

We note that
\[
\frac{\partial^2}{\partial\epsilon_1\partial\epsilon_2}\Lambda(\epsilon_1f^{(1)}+\epsilon_2f^{(2)})\vert_{\epsilon_1=\epsilon_2=0}=\nu\cdot S^L(\mathcal{U}^{(12)})+\nu\cdot G(u^{(1)},u^{(2)}).
\]
Assume $v$ solves the initial boundary value problem for the backward elastic wave equation
\begin{equation}\label{backward_eq}
\begin{split}
&\rho\frac{\partial^2}{\partial t^2}v-\nabla\cdot S^L(v)=0,\quad (t,x)\in (0,T)\times\Omega,\\
&v(t,x)=g,\quad (t,x)\in (0,T)\times\partial \Omega,\\
&v(T,x)=\frac{\partial}{\partial t}v(T,x)=0,\quad x\in \Omega.
\end{split}
\end{equation}
Using integration by parts, we get
\[
\begin{split}
&\int_0^T\int_{\partial\Omega}\left(\frac{\partial^2}{\partial\epsilon_1\partial\epsilon_2}\Lambda(\epsilon_1f_1+\epsilon_2f_2)\vert_{\epsilon_1=\epsilon_2=0}-\nu\cdot G(u^{(1)},u^{(2)})\right)g\,\mathrm{d}S\mathrm{d}t\\
=&\int_0^T\int_{\partial\Omega}\nu\cdot S^L(\mathcal{U}^{(12)})g\,\mathrm{d}S\mathrm{d}t\\
=&\int_0^T\int_{\Omega}\left(\nabla\cdot S^L(\mathcal{U}^{(12)})v+\mathbf{C}\nabla \mathcal{U}^{(12)}:\nabla v \right)\,\mathrm{d}x\mathrm{d}t\\
=&\int_0^T\int_{\Omega}\left(\rho\frac{\partial^2}{\partial t^2}\mathcal{U}^{(12)}-\nabla\cdot G(u^{(1)},u^{(2)})\right)v+\mathbf{C}\nabla \mathcal{U}^{(12)}:\nabla v \,\mathrm{d}x\mathrm{d}t\\
=&\int_0^T\int_{\Omega}\rho\mathcal{U}^{(12)}\frac{\partial^2}{\partial t^2}v-\nabla\cdot G(u^{(1)},u^{(2)})v+\mathbf{C}\nabla \mathcal{U}^{(12)}:\nabla v \,\mathrm{d}x\mathrm{d}t\\
=&\int_0^T\int_{\Omega}\mathcal{U}^{(12)}\rho\frac{\partial^2}{\partial t^2}v-\nabla\cdot G(u^{(1)},u^{(2)})v- \mathcal{U}^{(12)} \nabla \cdot S^L(v) \mathrm{d}x\mathrm{d}t+\int_0^T\int_{\partial\Omega}\nu\cdot S^L(v)\mathcal{U}^{(12)}\,\mathrm{d}S\mathrm{d}t\\
=&-\int_0^T\int_{\Omega}\nabla\cdot G(u^{(1)},u^{(2)})v\,\mathrm{d}x\mathrm{d}t.
\end{split}
\]
Here we use the notation $A:B=\sum_{i,j=1}^3A_{ij}B_{ij}$ for matrices $A$ and $B$.

Therefore, the displacement-to-traction map determines
\begin{equation}\label{integral_G}
\begin{split}
&\int_0^T\int_{\partial\Omega}\left(\frac{\partial^2}{\partial\epsilon_1\partial\epsilon_2}\Lambda(\epsilon_1f^{(1)}+\epsilon_2f^{(2)})\vert_{\epsilon_1=\epsilon_2=0}\right)g\,\mathrm{d}S\mathrm{d}t\\
=&-\int_0^T\int_{\Omega}\nabla\cdot G(u^{(1)},u^{(2)})v\,\mathrm{d}x\mathrm{d}t+\int_0^T\int_{\partial\Omega}\nu\cdot G(u^{(1)},u^{(2)})g\,\mathrm{d}S\mathrm{d}t\\
=&\int_0^T\int_{\Omega}\mathcal{G}(\nabla u^{(1)},\nabla u^{(2)},\nabla v)\,\mathrm{d}x\mathrm{d}t,
\end{split}
\end{equation}
where
\begin{equation}\label{integrand_G}
\begin{split}
&\mathcal{G}(\nabla u^{(1)},\nabla u^{(2)},\nabla v)\\
=&(\lambda+\mathscr{B})(\nabla u^{(1)}:\nabla u^{(2)})(\nabla\cdot v)+2\mathscr{C} (\nabla\cdot u^{(1)})(\nabla\cdot u^{(2)})(\nabla\cdot v)+\mathscr{B}(\nabla u^{(1)}:\nabla^T u^{(2)})(\nabla\cdot v)\\
&+\mathscr{B}\left( (\nabla\cdot u^{(1)})(\nabla u^{(2)}:\nabla^T v)+(\nabla\cdot u^{(2)})(\nabla u^{(1)}:\nabla^T v)\right)\\
&+\frac{\mathscr{A}}{4}\left(\frac{\partial u^{(1)}_j}{\partial x_m}\frac{\partial u^{(2)}_m}{\partial x_i}+\frac{\partial u^{(2)}_j}{\partial x_m}\frac{\partial u^{(1)}_m}{\partial x_i}\right)\frac{\partial v_i}{\partial x_j}\\
&+(\lambda+\mathscr{B})\left((\nabla\cdot u^{(1)})(\nabla u^{(2)}:\nabla v)+(\nabla\cdot u^{(2)})(\nabla u^{(1)}:\nabla v)\right)\\
&+\left(\mu+\frac{\mathscr{A}}{4}\right)\Bigg(\frac{\partial u^{(1)}_m}{\partial x_i}\frac{\partial u^{(2)}_m}{\partial x_j}+\frac{\partial u^{(2)}_m}{\partial x_i}\frac{\partial u^{(1)}_m}{\partial x_j}+\frac{\partial u^{(1)}_i}{\partial x_m}\frac{\partial u^{(2)}_j}{\partial x_m}+\frac{\partial u^{(2)}_i}{\partial x_m}\frac{\partial u^{(1)}_j}{\partial x_m}\\
&\quad\quad\quad\quad\quad\quad\quad\quad\quad+\frac{\partial u^{(1)}_i}{\partial x_m}\frac{\partial u^{(2)}_m}{\partial x_j}+\frac{\partial u^{(2)}_i}{\partial x_m}\frac{\partial u^{(1)}_m}{\partial x_j}\Bigg)\frac{\partial v_i}{\partial x_j}.
\end{split}
\end{equation}
Here we use the notation $\nabla^Tu=(\nabla u)^T$.\\

We will construct special solutions $u^{(1)}, u^{(2)},v$ for the linear elastic wave equation and recover the parameters $\mathscr{A},\mathscr{B},\mathscr{C}$ from the integral \eqref{integral_G}. We emphasize here that the solutions will be constructed with known coefficients $\lambda,\mu,\rho$ in the linearized equation.

\section{Gaussian beam solutions}\label{gaussianbeam}
Denote
\[
M=[0,T]\times\Omega.
\]
We note that $M$ can be viewed as a Lorentzian manifold with metric $-\mathrm{d}t^2+g_{P}$ or $-\mathrm{d}t^2+g_{S}$.

In this section, we will construct Gaussian beam solutions $u$ to the linear elastic wave equation
\begin{equation}\label{elastic_eq_linearized_2}
\begin{split}
&\rho\frac{\partial^2u}{\partial t^2}-\nabla\cdot S^L(x,u)=0,\quad (t,x)\in (0,T)\times\Omega,\\
&u(0,x)=\frac{\partial}{\partial t}u(0,x)=0,\quad x\in \Omega,
\end{split}
\end{equation}
of the form
\[
u(t,x)=e^{\mathrm{i}\varrho\varphi(t,x)}\mathfrak{a}(t,x)+R_\varrho(t,x),
\]
with a large parameter $\varrho$. The phase function $\varphi$ is complex-valued. The principal term $e^{\mathrm{i}\varrho\varphi(t,x)}\mathfrak{a}(t,x)$ is concentrated near a null geodesic $\vartheta$ in $(M,-\mathrm{d}t^2+g_{P/S})$. 
The remainder term $R_\varrho$ will vanish as $\varrho\rightarrow+\infty$.\\

For the construction of term $e^{\mathrm{i}\varrho\varphi(t,x)}\mathfrak{a}(t,x)$, we consider the equation
\[
\rho\frac{\partial^2u}{\partial t^2}-\nabla\cdot S^L(x,u)=0
\]
in an extended domain $\widetilde{\Omega}$, such that $\Omega\subset\subset \widetilde{\Omega}$. The parameters $\lambda,\mu,\rho$ are extended  smoothly to $\widetilde{\Omega}$. Also denote $\widetilde{M}=[0,T]\times \widetilde{\Omega}$.

We want to mention here that if either $(\Omega,g_{P/S})$ has no conjugate points or satisfies the foliation condition, then $(\Omega,g_{P/S})$ is non-trapping, i.e., every geodesic hits the boundary in finite time. See \cite[Proposition 3.31]{paternain_notes} and \cite[Lemma 2.1]{paternain2019geodesic}.\\

\noindent\textbf{Fermi coordinates.} We introduce Fermi coordinates in a neighborhood of the null geodesic $\vartheta$. Assume $\vartheta(t)=(t,\gamma(t))$, where $\gamma$ is a unit-speed geodesic in the Riemannian manifold $(\widetilde{\Omega},g)$, where $g=g_{P/S}$ is of interest to us. Assume $\vartheta$ passes through a point $(t_0,x_0)\in M$, i.e. $t_0\in (0,T)$ and $\gamma(t_0)=x_0\in\Omega$, and $\vartheta$ joins two points $(t_-,\gamma(t_-))$ and $(t_+,\gamma(t_+))$ where  $t_-,t_+\in (0,T)$ and $\gamma(t_-),\gamma(t_+)\in\partial\Omega$. Extend $\vartheta$ to $\widetilde{M}$ such that $\gamma(t)$ is well defined on $[t_--\epsilon,t_++\epsilon]\subset (0,T)$ with $\epsilon$ a small constant. We will follow the construction of the coordinates in \cite{feizmohammadi2019timedependent}. See also \cite{kurylev2014inverse}, \cite{uhlmann2018determination}.

Choose $\alpha_2,\alpha_3$ such that $\{\dot{\gamma}(t_0),\alpha_2,\alpha_3\}$ forms an orthonormal basis for $T_{x_0}\Omega$. Let $s$ denote the arc length along $\gamma$ from $x_0$. We note here that $s$ can be positive or negative, and $(t_0+s,\gamma(t_0+s))=\vartheta(t_0+s)$. For $k=2,3$, let $e_k(s)\in T_{\gamma(t_0+s)}\Omega$ be the parallel transport of $\alpha_k$ along $\gamma$ to the point $\gamma(t_0+s)$.

Define the coordinate system $(y^1=s,y^2,y^3)$ through $\mathcal{F}_1:\mathbb{R}^{3}\rightarrow \widetilde{\Omega}$:
\[
\mathcal{F}_1(s,y^2,y^3)=\exp_{\gamma(t_0+s)}\left(y^2e_2(s)+y^3e_3(s)\right).
\]
In the new coordinates, we have
\[
g\vert_{\gamma}=\sum_{j=1}^3(\mathrm{d}y^j)^2,\quad\text{and}\quad\frac{\partial g_{jk}}{\partial y^i}\Big\vert_\gamma=0,~1\leq i,j,k\leq 3.
\]
Then the Euclidean metric $g_E$ of $\mathbb{R}^3$ takes the form
\[
g_E=\sum_{1\leq i,j\leq 3}c^2g_{ij}\mathrm{d}y^i\mathrm{d}y^j.
\]
The Christoffel symbols then have the form
\begin{equation}\label{Christoffel}
\begin{split}
&\Gamma_{\alpha\beta}^1=-c^{-1}\frac{\partial c}{\partial s} g_{\alpha\beta},\quad \Gamma_{1\alpha}^\beta=\delta^\alpha_\beta c^{-1}\frac{\partial c}{\partial s},\quad \Gamma_{1\alpha}^1=c^{-1}\frac{\partial c}{\partial y^\alpha},\\
&\Gamma_{11}^\alpha=-c^{-1}g^{\alpha\beta}\frac{\partial c}{\partial y^\beta},\quad \Gamma_{11}^1=c^{-1}\frac{\partial c}{\partial s}.
\end{split}
\end{equation}
Here $\alpha,\beta\in \{2,3\}$ and $c=c_{P/S}$. 

On the Lorentzian manifold $(\widetilde{M},-\mathrm{d}t^2+ g)$, near the null geodesic $\vartheta:(t_--\frac{\epsilon}{2},t_++\frac{\epsilon}{2})\rightarrow \widetilde{M}$ where $\vartheta(t)=(t,\gamma(t))$, we introduce the Fermi coordinates,
\[
z^0=\tau=\frac{1}{\sqrt{2}}(t-t_0+s),\quad z^1=r=\frac{1}{\sqrt{2}}(-t+t_0+s),\quad z^j=y^j,\, j=2,3.
\]
Denote $\tau_\pm=\sqrt{2}(t_\pm-t_0)$.
Then on $\vartheta$ we have $\overline{g}=-\mathrm{d}t^2+ g$ satisfying
\[
\overline{g}\vert_\vartheta=2\mathrm{d}\tau\mathrm{d}r+\sum_{j=2}^3(\mathrm{d}z^j)^2\quad\text{and}\quad\frac{\partial \overline{g}_{jk}}{\partial z^i}\Big\vert_\vartheta=0,~0\leq i,j,k\leq 3.
\]
We will use the notations $z=(\tau,z')=(\tau=z^0,r=z^1,z'')$ and $y=(s=y^1,y').$\\

\noindent\textbf{Construction of Gaussian beams.} We will construct approximate Gaussian beam of order $N$ of the form
\[
u_\varrho=\mathfrak{a}e^{\mathrm{i}\varrho\varphi}
\]
with 
\[
\varphi=\sum_{k=0}^N\varphi_k(\tau,z'),\quad \mathfrak{a}(\tau,z')=\chi\left(\frac{|z'|}{\delta}\right)\sum_{k=0}^{N+1}\varrho^{-k}\mathbf{a}_k(\tau,z'),
\]
in a neighborhood of $\vartheta$,
\[
\mathcal{V}=\{(\tau,z')\in\widetilde{M}\vert \tau\in [\tau_--\frac{\epsilon}{\sqrt{2}},\tau_++\frac{\epsilon}{\sqrt{2}}],\,|z'|<\delta\}.
\]
Here $\delta>0$ is a small parameter. The smooth function $\chi:\mathbb{R}\rightarrow [0,+\infty)$ satisfies $\chi(t)=1$ for $|t|\leq\frac{1}{4}$ and $\chi(t)=0$ for $|t|\geq \frac{1}{2}$. We refer to \cite{feizmohammadi2019recovery} for more details. We denote $\mathbf{a}_0=\mathbf{a}=(a_1,a_2,a_3)$ and $\mathbf{a}_1=\mathbf{b}=(b_1,b_2,b_3)$. We note here that an extra term $\mathbf{a}_{N+1}$ is needed here to achieve order $N$ approximation, in compare with the solution constructed in \cite{feizmohammadi2019recovery}. The parameter $\delta$ is small such that $\mathfrak{a}\vert_{t=0}=\mathfrak{a}\vert_{t=T}=0$.\\

In a neighborhood of $\vartheta$, we calculate
\[
\begin{split}
\widetilde{\varepsilon}_{k\ell}(u_\varrho)=&\frac{1}{2}(a_{k;\,\ell}+a_{\ell;\,k})e^{\mathrm{i}\varrho\varphi}+\frac{1}{2}\varrho^{-1}(b_{k;\,\ell}+b_{\ell;\,k})e^{\mathrm{i}\varrho\varphi}\\
&+\frac{1}{2}\mathrm{i}\varrho(a_k\varphi_{;\,\ell}+a_\ell\varphi_{;\,k})e^{\mathrm{i}\varrho\varphi}+\frac{1}{2}\mathrm{i}(b_k\varphi_{;\,\ell}+b_\ell\varphi_{;\,k})e^{\mathrm{i}\varrho\varphi},
\end{split}
\]
and
\[
\begin{split}
\sigma_{ij}(u_\varrho):=S^L_{ij}(u_\varrho)=&\lambda e^{k\ell}\widetilde{\varepsilon}_{k\ell}e_{ij}+2\mu\widetilde{\varepsilon}_{ij}\\
=&\lambda c^{-2}g^{k\ell}\widetilde{\varepsilon}_{k\ell}c^2g_{ij}+2\mu\widetilde{\varepsilon}_{ij}\\
=&\mathrm{i}\varrho(\lambda a_k\varphi_{;\ell}g^{k\ell}g_{ij}+\mu a_i\varphi_{;\,j}+\mu a_j\varphi_{;\,i})e^{\mathrm{i}\varrho\varphi}\\
&+(\lambda a_{k;\ell}g^{k\ell}g_{ij}+\mathrm{i}\lambda b_k \varphi_{;\ell}g^{k\ell}g_{ij}+\mu (a_{i;\,j}+a_{j;\,i})+\mathrm{i}\mu(b_i\varphi_{;\,j}+b_j\varphi_{;\,i}))e^{\mathrm{i}\varrho\varphi}\\
&+\mathcal{O}(\varrho^{-1}).
\end{split}
\]
We proceed to calculate
\[
\begin{split}
\sigma_{ij;m}=&\partial_m\sigma_{ij}-\Gamma_{im}^n\sigma_{nj}-\Gamma_{jm}^n\sigma_{ni}\\
=&-\varrho^2(\lambda a_k\varphi_{;\ell}g^{k\ell}g_{ij}+\mu a_i\varphi_{;j}+\mu a_j\varphi_{;i})\varphi_{;m}e^{\mathrm{i}\varrho\varphi}\\
&+\mathrm{i}\varrho \partial_m(\lambda a_k\varphi_{;\ell}g^{k\ell}g_{ij}+\mu a_i\varphi_{;j}+\mu a_j\varphi_{;i})e^{\mathrm{i}\varrho\varphi}\\
&+\mathrm{i}\varrho\left(\lambda a_{k;\ell}g^{k\ell}g_{ij}+\mu(a_{i;j}+a_{j;i})+\mathrm{i} \lambda b_k\varphi_{;\ell}g^{k\ell}g_{ij}+\mathrm{i}\mu(b_i\varphi_{;j}+b_j\varphi_{;i})\right)\varphi_{;m}e^{\mathrm{i}\varrho\varphi}\\
&-\mathrm{i}\varrho\Gamma_{im}^n(\lambda a_k\varphi_{;\ell}g^{k\ell}g_{nj}+\mu a_n\varphi_{;j}+\mu a_j\varphi_{;n})e^{\mathrm{i}\varrho\varphi}\\
&-\mathrm{i}\varrho\Gamma_{jm}^n(\lambda a_k\varphi_{;\ell}g^{k\ell}g_{ni}+\mu a_n\varphi_{;i}+\mu a_i\varphi_{;n})e^{\mathrm{i}\varrho\varphi}+\mathcal{O}(1).
\end{split}
\]
and
\[
\begin{split}
(\nabla\cdot S^L)_i=\sigma_{ij;}^{\quad j}=&\sigma_{ij;m}g^{jm}c^{-2}\\
=&-\varrho^2(\lambda a_k\varphi_{;\ell}g^{k\ell}\varphi_{;i}+\mu a_i\varphi_{;j}\varphi_{;m}g^{jm}+\mu a_j\varphi_{;m}g^{jm}\varphi_{;i})c^{-2}e^{\mathrm{i}\varrho\varphi}\\
&+\mathrm{i}\varrho c^{-2}\left(\partial_i(\lambda a_k\varphi_{;\ell}g^{k\ell})+\lambda a_k\varphi_{;\ell}g^{k\ell}\partial_mg_{ij}g^{jm}+\partial_m(\mu a_i\varphi_{;j}+\mu a_j\varphi_{;i})g^{jm}\right)e^{\mathrm{i}\varrho\varphi}\\
&+\mathrm{i}\varrho c^{-2}\Big(\lambda a_{k;\ell}g^{k\ell}\varphi_{;i}+\mu(a_{i;j}+a_{j;i})\varphi_{;m}g^{jm}+\mathrm{i}\lambda b_k\varphi_{;\ell}g^{k\ell}\varphi_{;i}\\
&\quad\quad\quad\quad\quad\quad\quad\quad+\mathrm{i}\mu(b_i\varphi_{;j}\varphi_{;m}g^{jm}+\varphi_{;i}b_j\varphi_{;m}g^{jm})\Big)e^{\mathrm{i}\varrho\varphi}\\
&-\mathrm{i}\varrho c^{-2}\Gamma_{im}^ng^{mj}(\lambda a_k\varphi_{;\ell}g^{k\ell}g_{nj}+\mu a_n\varphi_{;j}+\mu a_j\varphi_{;n})e^{\mathrm{i}\varrho\varphi}\\
&-\mathrm{i}\varrho c^{-2}\Gamma_{jm}^ng^{mj}(\lambda a_k\varphi_{;\ell}g^{k\ell}g_{ni}+\mu a_n\varphi_{;i}+\mu a_i\varphi_{;n})e^{\mathrm{i}\varrho\varphi}+\mathcal{O}(1).
\end{split}
\]
We also calculate
\[
\begin{split}
\partial^2_t(u_\varrho)_i=&-\varrho^2(\partial_t\varphi)^2 a_i e^{\mathrm{i}\varrho\varphi}-\varrho(\partial_t\varphi)^2 b_ie^{\mathrm{i}\varrho\varphi}+\mathrm{i}\varrho\partial_t^2\varphi a_ie^{\mathrm{i}\varrho\varphi}+2\mathrm{i}\varrho(\partial_t\varphi)\partial_t a_ie^{\mathrm{i}\varrho\varphi}+\mathcal{O}(1).
\end{split}
\]
~\\

In a neighborhood of $\vartheta$, we can write
\begin{equation}\label{eq_phase}
\rho\partial_t^2 u_\varrho-\nabla\cdot S^L(u_\varrho)=e^{\mathrm{i}\varrho\varphi}\left(\varrho^2\mathcal{I}_1+\sum_{k=0}^N\varrho^{1-k}\mathcal{I}_{k+2}+\mathcal{O}(\varrho^{-N})\right),
\end{equation}
where
\[
\mathcal{I}_1=-\rho(\partial_t\varphi)^2\mathbf{a}+(\lambda+\mu)\langle\mathbf{a},\nabla\varphi\rangle\nabla\varphi+\mu|\nabla\varphi|^2\mathbf{a},
\]
or component-wisely
\[
(\mathcal{I}_1)_i=-\rho(\partial_t\varphi)^2a_i+\lambda a_j\varphi_{;\ell}g^{j\ell}c^{-2}\varphi_{;i}+\mu a_i\varphi_{;j}\varphi_{;\ell}c^{-2}g^{j\ell}+\mu a_j\varphi_{;\ell}g^{j\ell}c_P^{-2}\varphi_{;i},
\]
and
\[
\begin{split}
(\mathcal{I}_2)_i=&\rho(\partial^2_t\varphi)a_i+2\rho\partial_t\varphi\partial_ta_i+\mathrm{i}\rho(\partial_t\varphi)^2b_i\\
&- c_P^{-2}\left(\partial_i(\lambda a_k\varphi_{;\ell}g^{k\ell})+\lambda a_k\varphi_{;\ell}g^{k\ell}\partial_mg_{ij}g^{jm}+\partial_m(\mu a_i\varphi_{;j}+\mu a_j\varphi_{;i})g^{jm}\right)\\
&-c_P^{-2}\left(\lambda a_{k;\ell}g^{k\ell}\varphi_{;i}+\mu(a_{i;j}+a_{j;i})\varphi_{;m}g^{jm}+\mathrm{i}\lambda b_k\varphi_{;\ell}g^{k\ell}\varphi_{;i}+\mathrm{i}\mu(b_i\varphi_{;j}\varphi_{;m}g^{jm}+\varphi_{;i}b_j\varphi_{;m}g^{jm})\right)\\
&+\Gamma^n_{im}g^{mj}c_P^{-2}(\lambda a_k\varphi_{;\ell}g^{k\ell}g_{nj}+\mu a_n\varphi_{;j}+\mu a_j\varphi_{;n})\\
&+\Gamma^n_{jm}g^{mj}c_P^{-2}(\lambda a_k\varphi_{;\ell}g^{k\ell}g_{ni}+\mu a_n\varphi_{;i}+\mu a_i\varphi_{;n}).
\end{split}
\]

We will construct the phase function $\varphi$ and the amplitude $\mathfrak{a}$ such that
\begin{equation}\label{I1}
\frac{\partial^\Theta}{\partial z^\Theta}\mathcal{I}_k=0\text{ on } \vartheta
\end{equation}
for $\Theta=(0,\Theta_1,\Theta_2,\Theta_3)$ with $|\Theta|\leq N$ and $k=1,2,\cdots, N+2$. The detailed construction will be given later.

Assume that \eqref{I1} is satisfied, we construct the remainder term $R_\varrho$. We let $R_\varrho$ be the solution to the following initial boundary value problem
\begin{equation}\label{eq_remainder}
\begin{split}
&\rho\frac{\partial^2R_\varrho}{\partial t^2}-\nabla\cdot S^L(x,R_\varrho)=F_\varrho,\quad (t,x)\in (0,T)\times\Omega,\\
&R_\varrho=0,\quad \text{on } (0,T)\times\partial \Omega,\\
&R_\varrho(0,x)=\frac{\partial}{\partial t}R_\varrho(0,x)=0,\quad x\in \Omega.
\end{split}
\end{equation}
Here
\[
F_\varrho=-\rho\partial_t^2 u_\varrho+\nabla \cdot S^L(u_\varrho),
\]
in a neighborhood of $\vartheta$. 
By \eqref{I1} and \cite[Lemma 2]{feizmohammadi2019recovery}, we have
\[
\|F_\varrho\|_{H^k(M)}\leq C\varrho^{-K},
\]
where $K=\frac{N+1-k}{2}+1$.

By $L_2$ estimates for second order hyperbolic equation, we have
\[
\|R_\varrho\|_{H^{k+1}(M)}\leq C\|F_\varrho\|_{H^k(M)}.
\]
We can take $N$ large enough and use Sobolev imbedding to obtain
\begin{equation}\label{R_estimate}
\|R_\varrho\|_{W^{1,3}(M)}=\mathcal{O}(\varrho^{-1/2}).
\end{equation}

We remark here that $u=u_\varrho+R_\varrho$ solves the equation \eqref{elastic_linear_eq} with
\[
f=u_\varrho\vert_{[0,T]\times\partial\Omega}.
\]
~\\


\subsection{Construction of the phase}
We will construct phase function $\varphi=\varphi_{P/S}$ such that
\[
\mathcal{S}\varphi_P=(\lambda+2\mu)|\nabla\varphi_P|^2-\rho(\partial_t\varphi_P)^2\label{varphi_P},
\]
or
\[
\mathcal{S}\varphi_S=\mu|\nabla\varphi_S|^2-\rho(\partial_t\varphi_S)^2\label{varphi_S}
\]
vanishes on $\vartheta$ up to order $N$. In terms of Fermi coordinates $z=(z^0=\tau,z^1,z^2,z^3)$ for $\overline{g}_{P/S}=-\mathrm{d}t^2+g_{P/S}$ we need
\begin{equation}\label{S_cond}
\frac{\partial^\Theta}{\partial z^\Theta}(\mathcal{S}\varphi_{P/S})(\tau,0)=0
\end{equation}
for $\Theta=(0,\Theta_1,\Theta_2,\Theta_3)$ with $|\Theta|\leq N$.

Notice that \eqref{S_cond} is equivalent to
\[
\frac{\partial^\Theta}{\partial z^\Theta}\langle\mathrm{d}\varphi_\bullet,\mathrm{d}\varphi_\bullet\rangle_{\overline{g}_\bullet}\Big\vert_{\vartheta}=0.
\]
Thus the phase function $\varphi$ can be constructed as in \cite{feizmohammadi2019recovery} of the form
\[
\varphi=\sum_{k=0}^N\varphi_k(\tau,z').
\]
Here for each $k$, $\varphi_k$ is a complex valued homogeneous polynomial of degree $k$ with respect to the variables $z^i$, $i=1,2,3$. In this paper, we will use the explicit forms of $\varphi_0,\varphi_1,\varphi_2$, which will be constructed below. Following the lines in \cite{feizmohammadi2019timedependent}, one can take
\[
\varphi_0=0,\quad \varphi_1=r=\frac{-t+t_0+s}{\sqrt{2}},
\]
and
\[
\varphi_2(\tau,z')=\sum_{1\leq i,j\leq 3}H_{ij}(\tau)z^iz^j.
\]
Here $H$ is a symmetric matrix with $\Im H(\tau)>0$.

The matrix $H$ satisfies a Ricatti type ODE,
\begin{equation}\label{Ricatti}
\frac{\mathrm{d}}{\mathrm{d}\tau}H+HCH+D=0,\tau\in (\tau_--\frac{\epsilon}{2},\tau_++\frac{\epsilon}{2}),\quad H(0)=H_0,\text{ with }\Im H_0>0,
\end{equation}
where $C$, $D$ are matrices with $C_{11}=0$, $C_{ii}=2$, $i=2,3$, $C_{ij}=0$, $i\neq j$ and $D_{ij}=\frac{1}{4}(\partial_{ij}^2g^{11})$.
 
 \begin{lemma}[\text{\cite[Lemma 3.2]{feizmohammadi2019timedependent}}]
 The Ricatti equation \eqref{Ricatti} has a unique solution. Moreover the solution $H$ is symmetric and $\Im (H(\tau))>0$ for all $\tau\in (\tau_--\frac{\delta}{2},\tau_++\frac{\delta}{2})$. For solving the above Ricatti equation, one has
 \[
 H(\tau)=Z(\tau)Y(\tau)^{-1},
 \]
 where $Y(\tau)$ and $Z(\tau)$ solve the ODEs
 \[
 \begin{split}
& \frac{\mathrm{d}}{\mathrm{d}\tau}Y(\tau)=CZ(\tau),\quad Y(0)=Y_0,\\
&\frac{\mathrm{d}}{\mathrm{d}\tau}Z(\tau)=-D(\tau)Y(\tau),\quad Z(0)=Y_1=H_0Y_0.
 \end{split}
 \]
 In addition, $Y(\tau)$ is non-degenerate.
 \end{lemma}

 \begin{lemma}[\text{\cite[Lemma 3.3]{feizmohammadi2019timedependent}}]\label{lemma_H0}
 The following identity holds:
\[
 \det(\Im(H(\tau))|\det(Y(\tau))|^2=c_0
\]
 with $c_0$ independent of $\tau$.
 \end{lemma}
  We see that the matrix $Y(\tau)$ satisfies
  \begin{equation}\label{eq_Y}
  \frac{\mathrm{d}^2}{\mathrm{d}\tau^2}Y+CD Y=0,\quad Y(0)=Y_0,\quad \frac{\mathrm{d}}{\mathrm{d}\tau}Y(0)=CY_1.
  \end{equation}

\subsection{Construction of the amplitude for \textit{P}-waves}
We consider the Lorentzian manifold $(M,-\mathrm{d}t^2+g_P)$ and $\vartheta$ is a null-geodesic in it. For \textit{P}-waves, the polarization vector $\mathbf{a}$ should be in parallel with the wave vector $\nabla\varphi$ on $\vartheta$. Denote $\varphi=\varphi_P$ and take
\begin{equation}\label{a_form}
\mathbf{a}=A_P\nabla\varphi.
\end{equation}
Component-wisely, the gradient of $\varphi$ has the form
\[
 \varphi_{\,;\,1}\vert_\vartheta=\frac{1}{\sqrt{2}}, \quad  \varphi_{\,;\,\alpha}\vert_\vartheta=0.
\]
By \eqref{a_form}, we have
\[
a_1\vert_\vartheta=\frac{1}{\sqrt{2}}A_P,\quad a_\alpha\vert_\vartheta=0,
\]
and
\[
\mathcal{I}_1=(-\rho(\partial_t \varphi)^2+(\lambda+2\mu)|\nabla\varphi|^2)A_P\nabla\varphi.
\]
By the construction of the phase function $\varphi$, we have \eqref{I1} satisfied for $k=1$.\\

Next, we proceed to construct $A_P$. Let us first consider the equation \eqref{I1} with $k=2$, $\Theta=0$ and $i=1$. On $\vartheta$, we calculate
\begin{eqnarray*}
&&\rho(\partial^2_t\varphi)a_1=\frac{1}{\sqrt{2}}\rho (\partial^2_t\varphi)A_P,\\
&&2\rho(\partial_t \varphi)(\partial_t a_1)=-\sqrt{2}\rho\left(\frac{1}{\sqrt{2}}\partial_t A_P+A_P\frac{\partial^2\varphi}{\partial s\partial t}\right),\\
&&\mathrm{i}\rho(\partial_t\varphi)^2b_i=\frac{1}{2}\mathrm{i}\rho b_1,\\
&&\partial_1(\lambda a_k\varphi_{;\ell}g^{k\ell})=\frac{1}{2}\partial_s(\lambda A_P)+\sqrt{2}\lambda A_P\frac{\partial^2\varphi}{\partial s^2},\\
&&\lambda a_k\varphi_{;\ell}g^{k\ell}\partial_mg_{1j}g^{jm}=0,\\
&&\partial_m(\mu a_1\varphi_{;j}+\mu a_j\varphi_{;1})g^{jm}=\partial_s(\mu A_P)+2\sqrt{2}\mu A_P\frac{\partial^2\varphi}{\partial s^2}+\sqrt{2}\mu A_P\sum_{\alpha=2}^3\frac{\partial^2\varphi}{\partial y^\alpha\partial y^\alpha},\\
&&\lambda a_{k;\ell}g^{k\ell}\varphi_{;1}=\lambda\left(\frac{1}{\sqrt{2}}A_P\frac{\partial^2\varphi}{\partial s^2}+\frac{1}{2}\partial_sA_P+\frac{1}{2}A_Pc_P^{-1}\frac{\partial c_P}{\partial s}+\frac{1}{\sqrt{2}}A_P\sum_{\alpha=2}^3\frac{\partial^2\varphi}{\partial y^\alpha\partial y^\alpha}\right),\\
&&\mu(a_{1;j}+a_{j;1})\varphi_{;m}g^{jm}=\mu(\partial_sA_P+\sqrt{2}A_P\frac{\partial^2\varphi}{\partial s^2}-A_Pc_P\frac{\partial c_P}{\partial s}),\\
&&\mathrm{i}\lambda b_k\varphi_{;\ell}g^{k\ell}\varphi_{;1}+\mathrm{i}\mu(b_1\varphi_{;j}\varphi_{;m}g^{jm}+\varphi_{;1}b_j\varphi_{;m}g^{jm})=\frac{1}{2}\mathrm{i}(\lambda+2\mu)b_1,\\
&&\Gamma^n_{1m}g^{mj}(\lambda a_k\varphi_{;\ell}g^{k\ell}g_{nj}+\mu a_n\varphi_{;j}+\mu a_j\varphi_{;n})=(\frac{3}{2}\lambda+\mu)c_P^{-1}\frac{\partial c_P}{\partial s} A_P,\\
&&\Gamma^n_{jm}g^{mj}(\lambda a_k\varphi_{;\ell}g^{k\ell}g_{ni}+\mu a_n\varphi_{;i}+\mu a_i\varphi_{;n})=-(\frac{1}{2}\lambda+\mu)c_P^{-1}\frac{\partial c_P}{\partial s} A_P.
\end{eqnarray*}

Then we obtain the following equation on $\vartheta$:
\[
\begin{split}
&\frac{1}{\sqrt{2}}\left(\rho \partial^2_t\varphi-c_P^{-2}(\lambda+2\mu)\partial_s^2\varphi-c_P^{-2}(\lambda+2\mu)\sum_{\alpha=2}^3\frac{\partial^2\varphi}{\partial y^\alpha\partial y^\alpha}\right)A_P\\
-&\sqrt{2}c_P^{-2}(\lambda+2\mu)A_P\frac{\partial^2\varphi}{\partial s^2}-\sqrt{2}\rho A_P\frac{\partial^2\varphi}{\partial s\partial t}-c_P^{-2}(\lambda+2\mu)\partial_sA_P-\rho\partial_t A_P\\
+&\frac{1}{2}(\lambda+2\mu)c_P^{-3}\frac{\partial c_P}{\partial s}A_P-\frac{1}{2}c_P^{-2}\partial_s(\lambda+2\mu)A_P\\
+&\mathrm{i}\frac{1}{2}b_1[\rho-c_P^{-2}(\lambda+2\mu)]=0.
\end{split}
\]
Since $c_P^{-2}(\lambda+2\mu)=\rho$, $b_1$ can not be determined at this step, but will be determined from lower order asymptotics. Notice that on $\vartheta$,
\begin{equation}\label{YP}
\begin{split}
&\rho \partial^2_t\varphi-c_P^{-2}(\lambda+2\mu)\partial_s^2\varphi-c_P^{-2}(\lambda+2\mu)\sum_{\alpha=2}^3\frac{\partial^2\varphi}{\partial y^\alpha\partial y^\alpha}\\
=&\rho\left(\partial^2_t\varphi-\partial_s^2\varphi-\sum_{\alpha=2}^3\frac{\partial^2\varphi}{\partial y^\alpha\partial y^\alpha}\right)\\
=&\rho \square_{\overline{g}}\varphi\\
=&-\rho\sum_{\alpha=2}^3\frac{\partial^2\varphi}{\partial y^\alpha\partial y^\alpha}\\
=&-\rho\mathrm{Tr}(CH)=-\rho\frac{\partial}{\partial \tau}\log(\det(Y_P(\tau))),
\end{split}
\end{equation}
and, using $\frac{\partial}{\partial s}+\frac{\partial}{\partial t}=\sqrt{2}\frac{\partial}{\partial\tau}$,
\[
\begin{split}
&-\sqrt{2}c_P^{-2}(\lambda+2\mu)A_P\frac{\partial^2\varphi}{\partial s^2}-\sqrt{2}\rho A_P\frac{\partial^2\varphi}{\partial s\partial t}\\
=&-\sqrt{2}\rho A_P\left(\frac{\partial^2\varphi}{\partial s^2}+\frac{\partial^2\varphi}{\partial s\partial t}\right)\\
=&-2\rho A_P\frac{\partial}{\partial\tau}(\partial_s\varphi)\\
=&0.
\end{split}
\]

Then we arrive at the transport equation for the amplitude $A_P$ on $\vartheta$,
\begin{equation}\label{eq_AP}
2\frac{\partial A_P}{\partial \tau}+\left[\frac{1}{\lambda+2\mu}\frac{\partial (\lambda+2\mu)}{\partial \tau}-c_P^{-1}\frac{\partial c_P}{\partial \tau}+\frac{1}{\det(Y_P)}\frac{\partial \det(Y_P)}{\partial \tau}\right]A_P=0,
\end{equation}
or equivalently
\[
\frac{\partial}{\partial \tau}\ln \left[A_P^2\det(Y_P)c_P^{-1}(\lambda+2\mu)\right]=0.
\]
Then we can take
\[
A_P(\tau)=c\det(Y_P(\tau))^{-1/2}c_P(\tau,0)^{-1/2}\rho(\tau,0)^{-1/2},
\]
with some constant $c$.\\

Next we consider \eqref{I1} for $k=2$ and $i=\alpha=2,3$ and obtain an equation for $b_\alpha$, $\alpha=2,3$ on  the null geodesic $\vartheta$
\begin{equation}\label{balpha}
\mathrm{i}(\rho-c_P^{-2}\mu)b_\alpha-c_P^{-2}(\lambda+2\mu)\frac{\partial A_P}{\partial z^\alpha}+\mathfrak{F}(\varphi,A_P\vert_\vartheta)=0.
\end{equation}
We can get an expression for $b_\alpha$ on $\vartheta$.

Substituting the expression for $b_\alpha$ into the equation \eqref{I1} with $k=2$ and $|\Theta|=1$ and $i=1$, we end up with a transport equation for $\frac{\partial^\Theta}{\partial z^\Theta}A_P$ on $\vartheta$, from which we can determine the value of $\frac{\partial^\Theta}{\partial z^\Theta}A_P$. Then using again \eqref{balpha}, we can determine $b_\alpha$. Finally, the equation \eqref{I1} with $k=2$ and $|\Theta|=1$ and $\alpha=2,3$ gives us the value of $\frac{\partial^\Theta}{\partial z^\Theta}b_\alpha$ on $\vartheta$. Continuing with this process, we can have \eqref{I1} satisfied with $k=2$ and $|\Theta|\leq N$. 

The lower order terms $\mathbf{a}_k$, $k=1,2,\cdots, N+1$ in the amplitude $\mathfrak{a}$ can constructed as in \cite{feizmohammadi2019recovery} such that \eqref{I1} is satisfied for all $k=2,\cdots, N+2$ and $|\Theta|\leq N$. We note here that $(\mathbf{a}_{N+1})_1$ can take any value, while $(\mathbf{a}_{N+1})_\alpha$, $\alpha=2,3$ need to be determined.

\subsection{Construction of the amplitude for \textit{S}-waves}
Let us now consider a null geodesic $\vartheta$ in the Lorentzian manifold $(M,-\mathrm{d}t^2+g_S)$. Denote $\varphi=\varphi_S$. In the Fermi coordinates,
\[
 \varphi_{\,;\,1}\vert_\vartheta=\frac{1}{\sqrt{2}},\quad \varphi_{\,;\,\alpha}\vert_\vartheta=0,\text{ for }\alpha=2,3.
\]
Now
\[
\mathcal{I}_1=(\mathcal{S}\varphi)\mathbf{a}+(\lambda+\mu)\langle\mathbf{a},\nabla\varphi\rangle\nabla\varphi.
\]
For \textit{S}-waves, the polarization vector $\mathbf{a}$ should be perpendicular to the wave vector $\nabla\varphi$ on $\vartheta$.
In order for \eqref{I1} to hold, we also need
\begin{equation}\label{eq_a}
\frac{\partial^{\Theta}}{\partial z^{\Theta}}\langle \mathbf{a},\nabla\varphi\rangle\vert_\vartheta=0, 
\end{equation}
for $\Theta=(0,\Theta_1,\Theta_2,\Theta_3)$ with $|\Theta|\leq N$. For this we take 
\[
\mathbf{a}=A_S\mathbf{e}
\]
with $\mathbf{e}=(e_1,e_2,e_3)$ satisfying $|\mathbf{e}|=1$ on $\vartheta$ and
\begin{equation}\label{eq_e}
\frac{\partial^{\Theta}}{\partial z^{\Theta}}\langle \mathbf{e},\nabla\varphi\rangle\vert_\vartheta=0.
\end{equation}

First we construct such vector $\mathbf{e}$. Without loss of generality, we can fix an $\alpha\in\{2,3\}$ and let
\[
e_1\vert_\vartheta,\quad e_\alpha\vert_\vartheta=1,\quad e_{\alpha'}\vert_\vartheta=0,\text{ for }\alpha'\neq \alpha.
\]
 The equation \eqref{eq_e}  with $|\Theta|=1$ implies
\begin{equation}\label{eq_e1}
\frac{1}{\sqrt{2}}\frac{\partial e_1}{\partial z^k}+\frac{\partial^2\varphi}{\partial z^k\partial z^\alpha}=0,\text{ on }\vartheta
\end{equation}
for $k\in\{1,2,3\}$. Successively we can obtain, from \eqref{eq_e} for $|\Theta|\geq 2$, equations for $\frac{\partial^\Theta e_1}{\partial z^\Theta}$ that will be determined. From now on, we just assume $\mathbf{e}$ has already been chosen.

We first consider \eqref{I1} for $k=2$ and $i=\alpha$. We calculate
\begin{eqnarray*}
&&\rho(\partial^2_t\varphi)a_\alpha=\rho (\partial^2_t\varphi)A_S,\\
&&2\rho(\partial_t \varphi)(\partial_t a_\alpha)=-\sqrt{2}\rho\partial_tA_S-\sqrt{2}\rho A_S\frac{\partial e_\alpha}{\partial t},\\
&&\mathrm{i}\rho(\partial_t\varphi)^2b_\alpha=\frac{1}{2}\mathrm{i}\rho b_\alpha,\\
&&\partial_\alpha(\lambda a_k\varphi_{;\ell}g^{k\ell})=\lambda A_S\frac{\partial^2\varphi}{\partial y^\alpha\partial y^\alpha}+\frac{1}{\sqrt{2}}\lambda A_S\frac{\partial e_1}{\partial y^\alpha},\\
&&\lambda a_k\varphi_{;\ell}g^{k\ell}\partial_mg_{\alpha j}g^{jm}=0\\
&&\partial_m(\mu a_\alpha\varphi_{;j}+\mu a_j\varphi_{;\alpha})g^{jm}=\frac{1}{\sqrt{2}}\partial_s(\mu A_S)+\frac{1}{\sqrt{2}}\mu A_S\frac{\partial e_\alpha}{\partial s}+\mu A_S\left(\frac{\partial^2\varphi}{\partial s^2}+\frac{\partial^2\varphi}{\partial y^\alpha\partial y^\alpha}+\sum_{\beta=2}^3\frac{\partial^2\varphi}{\partial y^\beta\partial y^\beta}\right),\\
&&\lambda a_{k;\ell}g^{k\ell}\varphi_{;\alpha}=0,\\
&&\mu(a_{\alpha;j}+a_{j;\alpha})\varphi_{;m}g^{jm}=\frac{1}{\sqrt{2}}\mu\left(\frac{\partial A_S}{\partial s}+A_S\frac{\partial e_\alpha}{\partial s}+A_S\frac{\partial e_1}{\partial y^\alpha}-2c_S^{-1}\frac{\partial c_S}{\partial s}A_S\right),\\
&&\mathrm{i}\lambda b_k\varphi_{;\ell}g^{k\ell}\varphi_{;\alpha}+\mathrm{i}\mu(b_\alpha\varphi_{;j}\varphi_{;m}g^{jm}+\varphi_{;\alpha}b_j\varphi_{;m}g^{jm})=\frac{1}{2}\mathrm{i}\mu b_\alpha,\\
&&\Gamma^n_{\alpha m}g^{mj}(\lambda a_k\varphi_{;\ell}g^{k\ell}g_{nj}+\mu a_n\varphi_{;j}+\mu a_j\varphi_{;n})=0,\\
&&\Gamma^n_{jm}g^{mj}(\lambda a_k\varphi_{;\ell}g^{k\ell}g_{n\alpha}+\mu a_n\varphi_{;\alpha}+\mu a_\alpha\varphi_{;n})=-\frac{1}{\sqrt{2}}\mu c_S^{-1}\frac{\partial c_S}{\partial s} A_S.
\end{eqnarray*}
Then we obtain the following equation on $\vartheta$
\[
\begin{split}
&\left(\rho\partial_t^2\varphi-c_S^{-2}\mu\partial_s^2\varphi-c_S^{-2}\mu\sum_{\beta=2}^3\frac{\partial^2\varphi}{\partial y^\beta\partial y^\beta}\right)A_S\\
-&\sqrt{2}\left(\rho \frac{\partial e_\alpha}{\partial t}+c_s^{-2}\mu\frac{\partial e_\alpha}{\partial s}\right)A_S-\sqrt{2}(\rho\partial_tA_S+c_S^{-2}\mu\partial_sA_S)\\
+&\frac{1}{\sqrt{2}}\mu c_S^{-3}\frac{\partial c_S}{\partial s}A_S-\frac{1}{\sqrt{2}}c_S^{-2}\frac{\partial \mu}{\partial s}A_S+\frac{1}{2}\mathrm{i} b_\alpha(\rho-c_S^{-2}\mu)\\
-&c_S^{-2}(\lambda+\mu)A_S\left(\frac{1}{\sqrt{2}}\frac{\partial e_1}{\partial y^\alpha}+\frac{\partial^2\varphi}{\partial y^\alpha\partial y^\alpha}\right)=0.
\end{split}
\]
Since $\rho-c_S^{-2}\mu=0$, $b_\alpha$ can not be determined at this step. Notice
\[
\frac{1}{\sqrt{2}}\frac{\partial e_1}{\partial y^\alpha}+\frac{\partial^2\varphi}{\partial y^\alpha\partial y^\alpha}=0,
\]
on $\vartheta$ by setting $k=\alpha$ in \eqref{eq_e1} and
\[
\rho \frac{\partial e_\alpha}{\partial t}+c_s^{-2}\mu\frac{\partial e_\alpha}{\partial s}=\sqrt{2}\rho\frac{\partial e_\alpha}{\partial\tau}=0.
\]
Similar to \eqref{YP}, we have
\[
\rho\partial_t^2\varphi-c_S^{-2}\mu\partial_s^2\varphi-c_S^{-2}\mu\sum_{\beta=2}^3\frac{\partial^2\varphi}{\partial y^\beta\partial y^\beta}=-\rho\frac{\partial}{\partial \tau}\log(\det(Y_S(\tau))).
\]

Thus we end up with the following equation for the amplitude $A_S$ on $\vartheta$:
\begin{equation}\label{eq_AS}
2\frac{\partial A_S}{\partial \tau}+\left[\frac{1}{\mu}\frac{\partial \mu}{\partial \tau}-c_S^{-1}\frac{\partial c_S}{\partial \tau}+\frac{1}{\det(Y_S)}\frac{\partial \det(Y_S)}{\partial \tau}\right]A_S=0,
\end{equation}
The above equation is similar to the equation for $A_P$ \eqref{eq_AP}. Therefore we can take
\[
A_S(\tau)=c\det(Y_S(\tau))^{-1/2}c_S(\tau,0)^{-1/2}\rho(\tau,0)^{-1/2},
\]
with some constant $c$.\\

By calculation, we find that the equation $\mathcal{I}_2=0$ for $i=\alpha'$ always holds on $\vartheta$, with arbitrary choice of $b_{\alpha'}$. 

Next we consider the equation $(\mathcal{I}_2)_i=0$ for $i=1$.
We obtain the following equation on  the null geodesic $\vartheta$
\begin{equation}\label{b1}
\frac{1}{2}\mathrm{i}(\rho-c_S^{-2}(\lambda+2\mu))b_1-c_S^{-2}(\lambda+\mu)\frac{\partial A_S}{\partial z^\alpha}+\mathfrak{F}(\varphi,A_S\vert_\vartheta)=0.
\end{equation}
Substitute the expression for $b_1$ into the equation \eqref{I1} with $k=2$, $|\Theta|=1$ and $i=1$. We will end up with a transport equation for $\frac{\partial^\Theta}{\partial z^\Theta}A_S$ on $\vartheta$, from which we can determine the value of $\frac{\partial^\Theta}{\partial z^\Theta}A_S$. Then using again \eqref{b1}, we can determine $b_1$. Finally, the equation \eqref{I1} with $k=2$ and $|\Theta|=1$ gives us the value of $\frac{\partial^\Theta}{\partial z^\Theta}b_1$ on $\vartheta$.

As for \textit{P}-waves, we can construct lower order terms in the amplitude $\mathfrak{a}$ such that \eqref{I1} is satisfied.\\

\section{Proof of the main theorem}\label{main}
We will prove the uniqueness of $\mathscr{A},\,\mathscr{B},\,\mathscr{C}$ from the displacement-to-traction $\Lambda$ in this section. For the determination of $\mathscr{A},\,\mathscr{B}$, we will have a pointwise recovery in the interior $\Omega$.  With $\mathscr{A}$ and $\mathscr{B}$ determined, certain type of weighted ray transform of $\mathscr{C}$ (along any geodesic in $(\Omega,g_P)$) can be obtained from $\Lambda$. Under some geometric conditions, the weighted ray transform is invertible.
%
%

\subsection{Determination of $\mathscr{A}$ and $\mathscr{B}$}
 We introduce the notations
\[
\begin{split}
L^{S,*}_pM&=\{(\tau,\xi)\in T_p^*M,\, \tau^2=c_S^2|\xi|^2\},\\
L^{P,*}_pM&=\{(\tau,\xi)\in T_p^*M,\, \tau^2=c_P^2|\xi|^2\},
\end{split}
\]
for the \textit{S}-wave and \textit{P}-wave light cones at a point $p\in M$.
Let us start with a lemma used in \cite{de2018nonlinear}:
\begin{lemma}\label{xi_cc}
There exist nonzero $\zeta^{(2)},\zeta^{(0)}\in L^{S,*}_pM$, $\zeta^{(1)}\in L^{P,*}_pM$ such that $\zeta^{(0)}$, $\zeta^{(1)}$ and $\zeta^{(2)}$ are linearly dependent, while $\zeta^{(2)}$ and $\zeta^{(0)}$ are linearly independent.
\end{lemma}
For readers' convenience, we still include the proof here.
\begin{proof}
Assume $\zeta^{(k)}=(\tau^{(k)},\xi^{(k)})$, $k=1,2$. We have
\[
(\tau^{(1)})^2=c_P^2|\xi^{(1)}|^2,\quad (\tau^{(2)})^2=c_S^2|\xi^{(2)}|^2.
\]
Now we consider the vector $\zeta^{(0)}=a\zeta^{(1)}+b\zeta^{(2)}$. Without loss of generality, we can assume $a=1$, $|\xi^{(k)}|=1$ for $k=1,2$. In order for $\zeta^{(0)}\in L^{S,*}_pM$, we need
\[
(\tau^{(1)}+b\tau^{(2)})^2=c_S^2|\xi^{(1)}+b\xi^{(2)}|^2.
\]
Then we must have
\[
2b\sqrt{\mu(\lambda+2\mu)}-2\mu b\xi^{(1)}\cdot\xi^{(2)}+(\lambda+2\mu)=0.
\]
The above equation (with $b$ as the unknown) always has a nonzero solution. This finishes the proof of the lemma.

\end{proof}

Fix a point $x_0\in\Omega$. Let $p=(\frac{T}{2},x_0)\in M$, $\xi^{(0)},\xi^{(1)},\xi^{(2)}\in T_{x_0}^*\Omega$, $|\xi^{(k)}|=1$. By Lemma \ref{xi_cc}, we can choose $\xi^{(k)}$, $k=0,1,2$, such that
\[
\begin{split}
\zeta^{(1)}=(c_P,\xi^{(1)})\in L^{P,*}_pM,\\
\zeta^{(2)}= (c_S,\xi^{(2)})\in L^{S,*}_pM,\\
\zeta^{(0)}=  (c_S,\xi^{(0)})\in L^{S,*}_pM,
\end{split}
\]
satisfying
\begin{equation}\label{phase_sum}
\kappa_0\zeta^{(0)}+\kappa_1\zeta^{(1)}+\kappa_2\zeta^{(2)}=0.
\end{equation}

\begin{remark}
The only term with parameter $\mathscr{C}$ in \eqref{integrand_G} is $\mathscr{C} (\nabla\cdot u^{(1)})(\nabla\cdot u^{(2)})(\nabla\cdot v)$. If any of the three solutions $u^{(1)},u^{(2)}, v$ represents \textit{S}-wave, this term will essentially vanish. Thus for the recovery of $\mathscr{C}$, we need to take $u^{(1)},u^{(2)},v$ all representing \textit{P}-waves. However, one can not choose three vectors in  $L^{P,*}_pM$, such that they are linearly dependent but pairwise linearly independent. This explains why we need to recover $\mathscr{C}$ in a different way.
\end{remark}

Denote $\vartheta^{(0)}, \vartheta^{(2)}$ to be the null geodesics in Lorentzian manifold $(M,-\mathrm{d}t^2+g_S)$ with cotangent vector $\zeta^{(0)}$, $\zeta^{(2)}$ at point $p$, and $\vartheta^{(1)}$ the null geodesic in Lorentzian manifold $(M,-\mathrm{d}t^2+g_P)$ with cotangent vector $\zeta^{(1)}$ at point $p$. We will construct solutions:
\begin{itemize}
\item $u^{(1),\,P}_\varrho$  is the Gaussian beam solution representing \textit{P}-waves, concentrated near the null geodesic $\vartheta^{(1)}$;
\item $u^{(2),\,SV}_\varrho$  is the Gaussian beam solution representing \textit{SV}-waves, concentrated near the null geodesic $\vartheta^{(2)}$;
\item $v^{SV}_\varrho$  is the Gaussian beam solution representing \textit{SV}-waves, concentrated near the null geodesic $\vartheta^{(0)}$.
\end{itemize}

\begin{figure}[htbp]
\centering
\includegraphics[width=2 in]{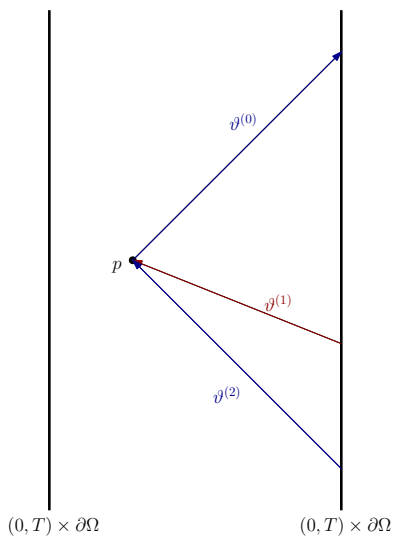}
\caption{Illustration of the choices of $\vartheta^{(1)},\vartheta^{(2)}$ and $\vartheta^{(0)}$.}
\end{figure}
More specifically, denote
\[
u^{(1),\,P}_\varrho=e^{\mathrm{i}\varrho\kappa_1\varphi^{(1)}}\chi^{(1)}(\mathbf{a}^{(1)}+\mathcal{O}(\varrho^{-1})),\quad u^{(2),\,SV}_\varrho=e^{\mathrm{i}\varrho\kappa_2\varphi^{(2)}}\chi^{(2)}(\mathbf{a}^{(2)}+\mathcal{O}(\varrho^{-1})),
\]
and
\[
v^{SV}_\varrho=e^{\mathrm{i}\varrho\kappa_0\varphi^{(0)}}\chi^{(0)}(\mathbf{a}^{(0)}+\mathcal{O}(\varrho^{-1})).
\]
By the construction of the phase functions $\varphi^{(k)}$, we can let
\[
\nabla\varphi^{(k)}(p)=\xi^{(k)},~\text{for}~k=0,1,2.
\]
The amplitudes can be chosen such that
\begin{equation}\label{choice_axi}
\mathbf{a}^{(1)}(p)=\xi^{(1)},\quad \mathbf{a}^{(2)}(p)=\mathbf{a}^{(0)}(p)\perp\text{span}\{\xi^{(1)},\xi^{(2)}\},\quad |\mathbf{a}^{(2)}(p)|=1.
\end{equation}
Similar to \cite[Lemma 5]{feizmohammadi2019recovery}, we have
 \begin{lemma}
 The function
 \[
 S:=\kappa_0\varphi^{(0)}+\kappa_1\varphi^{(1)}+\kappa_2\varphi^{(2)}
 \]
 is well-defined in a neighborhood of $p$ and
 \begin{enumerate}
 \item $S(p)=0$;
 \item $(\partial_t S(p),\nabla S(p))=0$;
 \item $\Im S(q)\geq cd(q,p)^2$ for $q$ in a neighborhood of $p$, where $c>0$ is a constant.
 \end{enumerate}
 \end{lemma}
 
 \begin{proof}
 The first claim is trivial since each of the three phases $\phi^{(k)}$ vanishes along $\vartheta^{(k)}$. For the second claim, one only needs to notice that 
 \[
 (\partial_t\varphi^{(0)},\nabla \varphi^{(0)})=(c_S,\xi^{(0)}),\quad  (\partial_t\varphi^{(1)},\nabla \varphi^{(1)})=(c_P,\xi^{(1)}),\quad(\partial_t\varphi^{(2)},\nabla \varphi^{(2)})=(c_S,\xi^{(2)})
 \]
 and use \eqref{phase_sum}.
 
 For the third claim, first notice $\Im\varphi^{(1)}(q)\geq 0$. We will prove $\Im\varphi^{(0)}(q)+\Im\varphi^{(2)}(q)\geq cd(p,q)$ for some constant $c>0$. Using Fermi coodinates for $(M,-\mathrm{d}t^2+g_S)$, we see that for $k=0,2$,
 \[
 \begin{split}
& D^2\Im\varphi^{(k)}(X,X)\geq 0,\quad\forall X\in T_pM,\\
&  D^2\Im\varphi^{(k)}(X,X)>0,\quad\forall X\in T_pM\setminus\text{span}(\xi^{(k),\sharp}).
 \end{split}
 \]
 Since $\xi^{(0)}$ and $\xi^{(2)}$ are linearly independent, the claim follows.
 \end{proof} 
 
 Now let $u^{(k)}$, $k=1,2$ to be the solution to \eqref{elastic_eq_linearized} with
 \[
 f^{(1)}=u^{(1),\,P}_\varrho\vert_{[0,T]\times\partial \Omega},\quad  f^{(2)}=u^{(2),\,SV}_\varrho\vert_{[0,T]\times\partial \Omega}
 \]
 and $v$ to be the solution to \eqref{backward_eq} with
  \[
 g=v^{SV}_\varrho\vert_{[0,T]\times\partial\Omega}.
 \]
 We remark here that $u^{(k),\,\bullet}_\varrho\vert_{t=0}=\frac{\partial}{\partial t}u_\varrho^{(k),\,\bullet}\vert_{t=0}=0$, and $v^{SV}_\varrho\vert_{t=T}=\frac{\partial}{\partial t}v^{SV}_\varrho\vert_{t=T}=0$ since $T>2\,\mathrm{diam}_S(\Omega)$.
 
 We will need the estimate \eqref{R_estimate} and the following ones
 \[
 \|u_{\varrho}^{(1),P}\|_{W^{1,3}}, \quad\|u_{\varrho}^{(2),SV}\|_{W^{1,3}}, \quad\|v_{\varrho}^{SV}\|_{W^{1,3}}=\mathcal{O}(\varrho^{1/2}).
 \]
Substituting $f^{(1)},f^{(2)},g$ constructed above in \eqref{integral_G}, we know that the displacement-to-traction map determines
\begin{equation}\label{integral}
\begin{split}
&\varrho^{-1}\frac{\mathrm{i}}{\kappa_1\kappa_2\kappa_3}\int_0^T\int_{\Omega}\mathcal{G}(\nabla u^{(1)},\nabla u^{(2)},\nabla v)\mathrm{d}x\mathrm{d}t\\
=&\varrho^{2}\int_Qe^{\mathrm{i}\varrho S}\chi^{(1)}\chi^{(2)}\chi^{(0)}\mathcal{G}(\mathbf{a}^{(1)}\otimes\nabla\varphi^{(1)},\mathbf{a}^{(2)}\otimes\nabla\varphi^{(2)},\mathbf{a}^{(0)}\otimes\nabla\varphi^{(0)})\mathrm{d}x\mathrm{d}t+\mathcal{O}(\varrho^{-1/2}).
\end{split}
\end{equation}

\textbf{First let us assume there are no conjugate points in $(\Omega,g_P)$. }Then the three null-geodesic $\vartheta^{(1)},\vartheta^{(2)},\vartheta^{(0)}$ intersect only at $p$, and thus the function $\chi^{(1)}\chi^{(2)}\chi^{(0)}$ is supported in a small neighborhood of $p$. Denote
\[
\mathcal{A}:=\mathcal{G}(\mathbf{a}^{(1)}\otimes\nabla\varphi^{(1)},\mathbf{a}^{(2)}\otimes\nabla\varphi^{(2)},\mathbf{a}^{(0)}\otimes\nabla\varphi^{(0)})
\]
with
\[
\begin{split}
&\mathcal{G}(\mathbf{a}^{(1)}\otimes\nabla\varphi^{(1)},\mathbf{a}^{(2)}\otimes\nabla\varphi^{(2)},\mathbf{a}^{(0)}\otimes\nabla\varphi^{(0)})\\
=&\mathscr{B}[(\mathbf{a}^{(1)}\cdot\nabla\varphi^{(1)})(\mathbf{a}^{(2)}\cdot\nabla\varphi^{(0)})(\mathbf{a}^{(0)}\cdot\nabla\varphi^{(2)})+(\mathbf{a}^{(2)}\cdot\nabla\varphi^{(2)})(\mathbf{a}^{(1)}\cdot\nabla\varphi^{(0)})(\mathbf{a}^{(0)}\cdot\nabla\varphi^{(2)})\\
&\quad\quad\quad\quad+ (\mathbf{a}^{(2)}\cdot\nabla\varphi^{(0)})(\mathbf{a}^{(1)}\cdot\nabla\varphi^{(0)})(\mathbf{a}^{(0)}\cdot\nabla\varphi^{(2)})]\\
&+\frac{\mathscr{A}}{4}\left((\mathbf{a}^{(2)}\cdot\nabla\varphi^{(1)})(\mathbf{a}^{(1)}\cdot\nabla\varphi^{(0)})(\mathbf{a}^{(0)}\cdot\nabla\varphi^{(2)})+(\mathbf{a}^{(1)}\cdot\nabla\varphi^{(2)})(\mathbf{a}^{(2)}\cdot\nabla\varphi^{(0)})(\mathbf{a}^{(0)}\cdot\nabla\varphi^{(1)})\right)\\
&+(\lambda+\mathscr{B})[(\mathbf{a}^{(1)}\cdot\nabla\varphi^{(1)})(\mathbf{a}^{(2)}\cdot\mathbf{a}^{(0)})(\nabla\varphi^{(2)}\cdot\nabla\varphi^{(0)})+(\mathbf{a}^{(2)}\cdot\nabla\varphi^{(2)})(\mathbf{a}^{(1)}\cdot\mathbf{a}^{(0)})(\nabla\varphi^{(1)}\cdot\nabla\varphi^{(0)})\\
&\quad\quad+(\mathbf{a}^{(1)}\cdot\mathbf{a}^{(2)})(\nabla\varphi^{(1)}\cdot\nabla\varphi^{(0)})(\mathbf{a}^{(0)}\cdot\nabla\varphi^{(0)})]+2\mathscr{C}(\mathbf{a}^{(1)}\cdot\nabla\varphi^{(1)})(\mathbf{a}^{(2)}\cdot\nabla\varphi^{(2)})(\mathbf{a}^{(0)}\cdot\nabla\varphi^{(0)})\\
&+(\mu+\frac{\mathscr{A}}{4})\Big((\mathbf{a}^{(1)}\cdot\mathbf{a}^{(2)})(\nabla\varphi^{(1)}\cdot\mathbf{a}^{(0)})(\nabla\varphi^{(2)}\cdot\nabla\varphi^{(0)})+(\nabla\varphi^{(1)}\cdot\nabla\varphi^{(2)})(\mathbf{a}^{(1)}\cdot\mathbf{a}^{(0)})(\mathbf{a}^{(2)}\cdot\nabla\varphi^{(0)})\\
&\quad \quad+(\mathbf{a}^{(2)}\cdot\mathbf{a}^{(1)})(\nabla\varphi^{(2)}\cdot\mathbf{a}^{(0)})(\nabla\varphi^{(1)}\cdot\nabla\varphi^{(0)})+(\mathbf{a}^{(2)}\cdot\mathbf{a}^{(0)})(\mathbf{a}^{(1)}\cdot\nabla\varphi^{(0)})(\nabla\varphi^{(1)}\cdot\nabla\varphi^{(2)})\\
&\quad\quad+(\nabla\varphi^{(1)}\cdot\mathbf{a}^{(2)})(\nabla\varphi^{(2)}\cdot\nabla\varphi^{(0)})(\mathbf{a}^{(1)}\cdot\mathbf{a}^{(0)})+(\nabla\varphi^{(2)}\cdot\mathbf{a}^{(1)})(\nabla\varphi^{(1)}\cdot\nabla\varphi^{(0)})(\mathbf{a}^{(2)}\cdot\mathbf{a}^{(0)})\Big).
\end{split}
\]
By the choice of $\alpha^{(k)}=\mathbf{a}^{(k)}(p)$ and $\xi^{(k)}$, $k=0,1,2$ in \eqref{choice_axi}. we have
\[
\begin{split}
\mathcal{A}(p)=&\mathscr{B}(x_0)[(\alpha^{(1)}\cdot\xi^{(1)})(\alpha^{(2)}\cdot\xi^{(0)})(\alpha^{(0)}\cdot\xi^{(2)})+(\alpha^{(2)}\cdot\xi^{(2)})(\alpha^{(1)}\cdot\xi^{(0)})(\alpha^{(0)}\cdot\xi^{(2)})\\
&\quad\quad\quad\quad+ (\alpha^{(2)}\cdot\xi^{(0)})(\alpha^{(1)}\cdot\xi^{(0)})(\alpha^{(0)}\cdot\xi^{(2)})]\\
&+\frac{\mathscr{A}}{4}(x_0)\left((\alpha^{(2)}\cdot\xi^{(1)})(\alpha^{(1)}\cdot\xi^{(0)})(\alpha^{(0)}\cdot\xi^{(2)})+(\alpha^{(1)}\cdot\xi^{(2)})(\alpha^{(2)}\cdot\xi^{(0)})(\alpha^{(0)}\cdot\xi^{(1)})\right)\\
&+(\lambda+\mathscr{B})(x_0)[(\alpha^{(1)}\cdot\xi^{(1)})(\alpha^{(2)}\cdot\alpha^{(0)})(\xi^{(2)}\cdot\xi^{(0)})+(\alpha^{(2)}\cdot\xi^{(2)})(\alpha^{(1)}\cdot\alpha^{(0)})(\xi^{(1)}\cdot\xi^{(0)})\\
&\quad\quad+(\alpha^{(1)}\cdot\alpha^{(2)})(\xi^{(1)}\cdot\xi^{(0)})(\alpha^{(0)}\cdot\xi^{(0)})]+2\mathscr{C}(x_0)(\alpha^{(1)}\cdot\xi^{(1)})(\alpha^{(2)}\cdot\xi^{(2)})(\alpha^{(0)}\cdot\xi^{(0)})\\
&+(\mu+\frac{\mathscr{A}}{4})(x_0)\Big((\alpha^{(1)}\cdot\alpha^{(2)})(\xi^{(1)}\cdot\alpha^{(0)})(\xi^{(2)}\cdot\xi^{(0)})+(\xi^{(1)}\cdot\xi^{(2)})(\alpha^{(1)}\cdot\alpha^{(0)})(\alpha^{(2)}\cdot\xi^{(0)})\\
&\quad \quad+(\alpha^{(2)}\cdot\alpha^{(1)})(\xi^{(2)}\cdot\alpha^{(0)})(\xi^{(1)}\cdot\xi^{(0)})+(\alpha^{(2)}\cdot\alpha^{(0)})(\alpha^{(1)}\cdot\xi^{(0)})(\xi^{(1)}\cdot\xi^{(2)})\\
&\quad\quad+(\xi^{(1)}\cdot\alpha^{(2)})(\xi^{(2)}\cdot\xi^{(0)})(\alpha^{(1)}\cdot\alpha^{(0)})+(\xi^{(2)}\cdot\alpha^{(1)})(\xi^{(1)}\cdot\xi^{(0)})(\alpha^{(2)}\cdot\alpha^{(0)})\Big)\\
=& (\lambda+\mathscr{B})(x_0)\xi^{(2)}\cdot\xi^{(0)}+(2\mu+\frac{\mathscr{A}}{2})(x_0)(\xi^{(1)}\cdot \xi^{(2)})(\xi^{(1)}\cdot \xi^{(0)}).
\end{split}
\]

Apply the method of stationary phase (cf., for example, \cite[Theorem 7.7.5]{hormander2015analysis}) to \eqref{integral}, we can recover
\[
(\lambda+\mathscr{B})(x_0)\xi^{(2)}\cdot\xi^{(0)}+(2\mu+\frac{\mathscr{A}}{2})(x_0)(\xi^{(1)}\cdot \xi^{(2)})(\xi^{(1)}\cdot \xi^{(0)}),
\]
by taking $\varrho\rightarrow+\infty$.
This is exactly the same quantity recovered in \cite{de2018nonlinear} using the nonlinear interaction of distorted plane \textit{P} and \textit{SV} waves.
By varying $\xi^{(1)},\xi^{(2)}$ ($\xi^{(0)}$ will be varying accordingly), we can recover $\lambda+\mathscr{B}$ and $2\mu+\frac{\mathscr{A}}{2}$ separately at the point $x_0$. Since $x_0$ can be any point in $\Omega$, this completes the determination of $\mathscr{A}$ and $\mathscr{B}$ in $\Omega$.\\

\begin{remark}
We only used the nonlinear interaction of \textit{P} and \textit{SV} waves to determine $\mathscr{A}$ and $\mathscr{B}$. It has already been observed to be possible in \cite{de2018nonlinear}, wherein nonlinear interactions of other types are analyzed as well.
\end{remark}

\textbf{Now assume $(\Omega,g_P)$ satisfies the foliation condition.} As in \cite{uhlmann2016inverse}, for any point $q\in\partial\Omega$, there exists a wedge-shaped neighborhood $O_q\subset \Omega$ of $q$ such that any geodesic in $(O_q,g_P)$ has no conjugate points. Then the three null-geodesics $\vartheta^{(1)},\vartheta^{(2)},\vartheta^{(0)}$, if $\vartheta^{(1)}\subset ((0,T)\times O_q,-\mathrm{d}t^2+g_P)$, can intersect only at $p$ (since $c_P>c_S$). We can now recover $\mathscr{A}$ and $\mathscr{B}$ in $O_q$. Then the foliation condition allows a layer stripping scheme to recover the two parameters in the whole domain $\Omega$. For more details, we refer to \cite{uhlmann2016inverse}.
\subsection{Determination of $\mathscr{C}$}
Finally, we recover the parameter $\mathscr{C}$. 

Since $\mathscr{A}$ and $\mathscr{B}$ have already been determined, the displacement-to-traction map now gives
\begin{equation}\label{integral_identity}
\int_0^T\int_{\Omega}\mathscr{C} (\nabla\cdot u^{(1)})(\nabla\cdot u^{(2)})(\nabla\cdot v)\mathrm{d}x\mathrm{d}t.
\end{equation}
We will construct $u^{(1)},u^{(2)}$ and $v$ so that they all represent \textit{P}-waves. Still use Gaussian beam solutions, now concentrating near the same null geodesic $\vartheta$ in $(M,-\mathrm{d}t^2+ g_P)$. Let
\[
u^{(1)}_\varrho=u^{(2)}_\varrho=e^{\mathrm{i}\varrho\varphi}\chi(\frac{|z'|}{\delta})(\mathbf{a}+\mathcal{O}(\varrho^{-1})),
\]
\[
v_\varrho=e^{-2\mathrm{i}\varrho\overline{\varphi}}\chi(\frac{|z'|}{\delta})(\overline{\mathbf{a}}+\mathcal{O}(\varrho^{-1})).
\]
let $u^{(k)}$, $k=1,2$ to be the solution to \eqref{elastic_eq_linearized} with
 \[
 f^{(1)}=u^{(1)}_\varrho\vert_{[0,T]\times\partial \Omega},\quad  f^{(2)}=u^{(2)}_\varrho\vert_{[0,T]\times\partial \Omega}
 \]
 and $v$ to be the solution to \eqref{backward_eq} with
  \[
 g=v_\varrho\vert_{[0,T]\times\partial \Omega}.
 \]

We extend $\mathscr{C}$ to $\widetilde{\Omega}$ such that $\mathscr{C}=0$ in $\widetilde{\Omega}\setminus \Omega$. Then the displacement-to-traction map determines
\[
\begin{split}
&\frac{1}{2\mathrm{i}}\varrho^{-3/2}\int_M \mathscr{C}(\nabla\cdot u^{(1)})(\nabla\cdot u^{(2)})(\nabla\cdot v)\mathrm{d}V\\
=&\varrho^{3/2}\int_{\tau_--\frac{\epsilon}{\sqrt{2}}}^{\tau_++\frac{\epsilon}{\sqrt{2}}}\int_{|z'|<\delta}\mathscr{C} e^{-4\varrho\Im\varphi}\chi^3(\frac{|z'|}{\delta})(\nabla\varphi\cdot \mathbf{a})(\nabla\varphi\cdot  \mathbf{a})(\overline{\nabla\varphi\cdot  \mathbf{a}})c_P^{3}\mathrm{d}z'\wedge\mathrm{d}\tau+\mathcal{O}(\varrho^{-1}).
\end{split}
\]
with $\delta$ sufficiently small.
Notice
\[
(\nabla\varphi\cdot \mathbf{a})(\nabla\varphi\cdot  \mathbf{a})(\overline{\nabla\varphi\cdot  \mathbf{a}})\vert_{\vartheta(\tau)}=c|\det Y(\tau)|^{-1}(\det Y(\tau))^{-1/2}c_P^{-15/2}(\tau)\rho^{-3/2}(\tau),
\]
where $c$ is some constant.
Thus, using method of stationary phase and Lemma \ref{lemma_H0}, we have
\[
\begin{split}
&\lim_{\varrho\rightarrow+\infty}\varrho^{3/2}\int_{|z'|<\delta} \mathscr{C}e^{-4\varrho\Im\varphi}\chi^3(\frac{|z'|}{\delta})(\nabla\varphi\cdot \mathbf{a})(\nabla\varphi\cdot  \mathbf{a})(\overline{\nabla\varphi\cdot  \mathbf{a}})c_P^{3}\mathrm{d}z'\\
=&c\mathscr{C}(\tau,0)c_P^{-9/2}(\tau,0)\rho(\tau,0)^{-3/2}(\det Y(\tau))^{-1/2}.
\end{split}
\]
Thus we can recover
\[
\int_{\vartheta}\mathscr{C}c_P^{-9/2}(\tau,0)\rho^{-3/2}(\tau,0)(\det Y(\tau))^{-1/2}\mathrm{d}\tau.
\]

Remember that $Y$ solves the equation \eqref{eq_Y}. By \cite[Corollary 3.5]{feizmohammadi2019timedependent},
\[
\frac{\partial^2\overline{g}^{11}}{\partial z^1\partial z^j}\Big\vert_\vartheta=0.
\]
 Thus one can take $Y_{11}=c_0$, $Y_{1j}=Y_{j1}=0$.  Here $c_0>0$ is independent of $\tau$. Denote $\widetilde{Y}=(Y_{\alpha\beta})_{\alpha,\beta=2}^3$. Then the $2\times 2$ matrix $\widetilde{Y}$ satisfies
 \[
 \frac{\mathrm{d}^2}{\mathrm{d}\tau^2}\widetilde{Y}+\widetilde{D}\widetilde{Y}=0,\quad \widetilde{Y}(a)=\widetilde{Y}_0,\quad \frac{\mathrm{d}}{\mathrm{d}\tau}\widetilde{Y}(a)=\widetilde{Y}_1,
 \]
 where $\widetilde{D}_{\alpha\beta}=\frac{1}{2}\left(\partial_{\alpha\beta}^2g^{11}\vert_\vartheta\right)_{\alpha,\beta=2}^3$.\\

 Now let us use some notations and definitions introduced in \cite{dahl2009geometrization}. We will follow the lines in \cite{feizmohammadi2019inverse}. Assume $\vartheta(t)=(t,\gamma(t))$ where $\gamma$ is a geodesic in the Riemannian manifold $(\Omega,g_P)$. Denote
 \[
 \dot{\gamma}(t)^\perp:=\{v\in T_{\gamma(t)}\Omega\vert g_P(\dot{\gamma}(t),v)=0\}
 \]
 to be the orthogonal complement at the point $\gamma(t)$ of $\dot{\gamma}$. Define the $(1,1)$-tensor $\Pi_{\gamma(t)}$ to be the projection from $T_{\gamma(t)}M$ onto $ \dot{\gamma}(t)^\perp$. A $(1,1)$-tensor $L(t)$ is said  to be transversal if $\Pi_{\gamma(t)} L(t)\Pi_{\gamma(t)}=L(t)$. Denote $\mathbb{Y}_\gamma$ to be the set of all transversal $(1,1)$-tensors $Y(t)$ that solve the complex Jacobi equation
\[
\frac{\mathrm{d}^2}{\mathrm{d}t^2}Y(t)-K(t)Y(t)=0,
\]
 subject to the constraint that
 \[
 \text{$Y(t_0)$ is non-degenerate, $\dot{Y}(t_0)Y(t_0)^{-1}$ is symmetric and $\Im(\dot{Y}(t_0)Y(t_0)^{-1})>0$}.
\]
Here $K=K_j^i\frac{\partial}{\partial y^i}\otimes\mathrm{d}y^j$, $K_j^i=g^{ik}R_{kj}$, $R$ is the Ricci tensor.
 
As in \cite{feizmohammadi2019timedependent}, we now have the \textit{Jacobi weighted ray transform of the first kind} (cf. \cite{feizmohammadi2019inverse}) of $\mathscr{C}c_P^{-9/2}\rho^{-3/2}=:f$ along the geodesic $\gamma$ in $(\Omega,g_P)$ passing through $x_0=\gamma(t_0)$,
\[
\mathscr{J}_Y^{(1)}f=\int_{t_-}^{t_+}f(\gamma(t))(\det Y(t))^{-1/2}\mathrm{d}t
\]
for any $Y\in\mathbb{Y}_\gamma$.\\

By \cite[Proposition 3]{feizmohammadi2019inverse}, $\mathscr{J}_Y^{(1)}f$ uniquely determines $f(x_0)$ \textbf{if $(\Omega,g_P)$ has no conjugate points}. Therefore, we can recover $\mathscr{C}(x_0)$ since $c_P$ and $\rho$ are already known. Here we use only the ray transform along this single geodesic. It is possible since we have the integral of $f$ along the geodesic with a class of weights. \\

\textbf{If $(\Omega,g_P)$ satisfies the foliation condition}, we can adopt a layer stripping method as previously used. One can also directly use the the invertibility of weighted geodesic ray transform with a single weight established in \cite{paternain2019geodesic}.

\subsection*{Acknowledgements}
GU was partially supported by NSF, a Walker Professorship at UW and a Si-Yuan Professorship at IAS, HKUST. JZ acknowledges the great hospitality of MSRI, where part of this work is carried out during his visit.\bibliographystyle{abbrv}
\bibliography{biblio}

\begin{thebibliography}{10}

\bibitem{assylbekov2017direct}
Y.~M. Assylbekov and T.~Zhou.
\newblock Direct and inverse problems for the nonlinear time-harmonic {M}axwell
  equations in {K}err-type media.
\newblock {\em arXiv preprint arXiv:1709.07767}, 2017.

\bibitem{bao2014sensitivity}
G.~Bao and H.~Zhang.
\newblock Sensitivity analysis of an inverse problem for the wave equation with
  caustics.
\newblock {\em Journal of the American Mathematical Society}, 27(4):953--981,
  2014.

\bibitem{belishev1996boundary}
M.~Belishev and A.~Katchalov.
\newblock Boundary control and quasiphotons in the problem of reconstruction of
  a {R}iemannian manifold via dynamic data.
\newblock {\em Journal of Mathematical Sciences}, 79(4):1172--1190, 1996.

\bibitem{bhattacharyya2018local}
S.~Bhattacharyya.
\newblock Local uniqueness of the density from partial boundary data for
  isotropic elastodynamics.
\newblock {\em Inverse Problems}, 34(12):125001, 2018.

\bibitem{carstea2019reconstruction}
C.~I. C{\^a}rstea, G.~Nakamura, and M.~Vashisth.
\newblock Reconstruction for the coefficients of a quasilinear elliptic partial
  differential equation.
\newblock {\em Applied Mathematics Letters}, 2019.

\bibitem{chen2019detection}
X.~Chen, M.~Lassas, L.~Oksanen, and G.~P. Paternain.
\newblock Detection of {H}ermitian connections in wave equations with cubic
  non-linearity.
\newblock {\em arXiv preprint arXiv:1902.05711}, 2019.

\bibitem{dahl2009geometrization}
M.~F. Dahl.
\newblock Geometrization of the leading term in acoustic {G}aussian beams.
\newblock {\em Journal of Nonlinear Mathematical Physics}, 16(01):35--45, 2009.

\bibitem{dairbekov2006integral}
N.~S. Dairbekov.
\newblock Integral geometry problem for nontrapping manifolds.
\newblock {\em Inverse Problems}, 22(2):431, 2006.

\bibitem{de2018nonlinear}
M.~V. de~Hoop, G.~Uhlmann, and Y.~Wang.
\newblock Nonlinear interaction of waves in elastodynamics and an inverse
  problem.
\newblock {\em Mathematische Annalen}, pages 1--31, 2018.

\bibitem{de2003finite}
W.~De~Lima and M.~Hamilton.
\newblock Finite-amplitude waves in isotropic elastic plates.
\newblock {\em Journal of sound and vibration}, 265(4):819--839, 2003.

\bibitem{dos2016calderon}
D.~Dos Santos~Ferreira, Y.~Kurylev, M.~Lassas, and M.~Salo.
\newblock The {C}alder{\'o}n problem in transversally anisotropic geometries.
\newblock {\em Journal of the European Mathematical Society},
  18(11):2579--2626, 2016.

\bibitem{feizmohammadi2019timedependent}
A.~Feizmohammadi, J.~Ilmavirta, Y.~Kian, and L.~Oksanen.
\newblock Recovery of time dependent coefficients from boundary data for
  hyperbolic equations.
\newblock {\em arXiv preprint arXiv:1901.04211}, 2019.

\bibitem{feizmohammadi2019recovery}
A.~Feizmohammadi and L.~Oksanen.
\newblock Recovery of zeroth order coefficients in non-linear wave equations.
\newblock {\em arXiv preprint arXiv:1903.12636}, 2019.

\bibitem{feizmohammadi2019inverse}
A.~Feizmohammadi and L.~Oksanen.
\newblock An inverse problem for a semi-linear elliptic equation in riemannian
  geometries.
\newblock {\em Journal of Differential Equations}, 269(6):4683--4719, 2020.

\bibitem{hansen2003propagation}
S.~Hansen and G.~Uhlmann.
\newblock Propagation of polarization in elastodynamics with residual stress
  and travel times.
\newblock {\em Mathematische Annalen}, 326(3):563--587, 2003.

\bibitem{herglotz1905elastizitaet}
G.~Herglotz.
\newblock {\"U}ber die {E}lastizitaet der {E}rde bei {B}eruecksichtigung ihrer
  variablen {D}ichte.
\newblock {\em Zeitschr. f{\"u}r Math. Phys}, 52:275--299, 1905.

\bibitem{hormander2015analysis}
L.~H{\"o}rmander.
\newblock {\em The analysis of linear partial differential operators {I}:
  {D}istribution theory and {F}ourier analysis}.
\newblock Springer, 2015.

\bibitem{kang2002identification}
H.~Kang and G.~Nakamura.
\newblock Identification of nonlinearity in a conductivity equation via the
  {D}irichlet-to-{N}eumann map.
\newblock {\em Inverse Problems}, 18(4):1079, 2002.

\bibitem{katchalov1998multidimensional}
A.~Katchalov and Y.~Kurylev.
\newblock Multidimensional inverse problem with incomplete boundary spectral
  data.
\newblock {\em Communications in Partial Differential Equations},
  23(1-2):27--59, 1998.

\bibitem{krupchyk2019partial}
K.~Krupchyk and G.~Uhlmann.
\newblock Partial data inverse problems for semilinear elliptic equations with
  gradient nonlinearities.
\newblock {\em arXiv preprint arXiv:1909.08122}, 2019.

\bibitem{kurylev2014inverse}
Y.~Kurylev, M.~Lassas, L.~Oksanen, and G.~Uhlmann.
\newblock Inverse problem for {E}instein-scalar field equations.
\newblock {\em arXiv preprint arXiv:1406.4776}, 2014.

\bibitem{kurylev2018inverse}
Y.~Kurylev, M.~Lassas, and G.~Uhlmann.
\newblock Inverse problems for {L}orentzian manifolds and non-linear hyperbolic
  equations.
\newblock {\em Inventiones mathematicae}, 212(3):781--857, 2018.

\bibitem{landau1960theory}
L.~Landau, E.~Lifshitz, J.~Sykes, W.~Reid, and E.~H. Dill.
\newblock Theory of elasticity: {V}ol. 7 of course of theoretical physics.
\newblock {\em Phys. Today}, 13:44, 1960.

\bibitem{lassas2019inverse}
M.~Lassas, T.~Liimatainen, Y.-H. Lin, and M.~Salo.
\newblock Inverse problems for elliptic equations with power type
  nonlinearities.
\newblock {\em arXiv preprint arXiv:1903.12562}, 2019.

\bibitem{lassas2019partial}
M.~Lassas, T.~Liimatainen, Y.-H. Lin, and M.~Salo.
\newblock Partial data inverse problems and simultaneous recovery of boundary
  and coefficients for semilinear elliptic equations.
\newblock {\em arXiv preprint arXiv:1905.02764}, 2019.

\bibitem{lassas2018inverse}
M.~Lassas, G.~Uhlmann, and Y.~Wang.
\newblock Inverse problems for semilinear wave equations on {L}orentzian
  manifolds.
\newblock {\em Communications in Mathematical Physics}, 360(2):555--609, 2018.

\bibitem{paternain2015invariant}
G.~P. Paternain, M.~Salo, and G.~Uhlmann.
\newblock Invariant distributions, {B}eurling transforms and tensor tomography
  in higher dimensions.
\newblock {\em Mathematische Annalen}, 363(1-2):305--362, 2015.

\bibitem{paternain_notes}
G.~P. Paternain, M.~Salo, and G.~Uhlmann.
\newblock Geometric inverse problems in 2{D}, 2020.

\bibitem{paternain2019geodesic}
G.~P. Paternain, M.~Salo, G.~Uhlmann, and H.~Zhou.
\newblock The geodesic {X}-ray transform with matrix weights.
\newblock {\em American Journal of Mathematics}, 141(6):1707--1750, 2019.

\bibitem{rachele2000inverse}
L.~Rachele.
\newblock An inverse problem in elastodynamics: uniqueness of the wave speeds
  in the interior.
\newblock {\em Journal of Differential Equations}, 162(2):300--325, 2000.

\bibitem{rachele2003uniqueness}
L.~Rachele.
\newblock Uniqueness of the density in an inverse problem for isotropic
  elastodynamics.
\newblock {\em Transactions of the American Mathematical Society},
  355(12):4781--4806, 2003.

\bibitem{sharafutdinov1992integral}
V.~A. Sharafutdinov.
\newblock Integral geometry of a tensor field on a manifold with upper-bounded
  curvature.
\newblock {\em Sibirskii Matematicheskii Zhurnal}, 33(3):192--204, 1992.

\bibitem{stefanov2016boundary}
P.~Stefanov, G.~Uhlmann, and A.~Vasy.
\newblock Boundary rigidity with partial data.
\newblock {\em Journal of the American Mathematical Society}, 29(2):299--332,
  2016.

\bibitem{SUV3}
P.~Stefanov, G.~Uhlmann, and A.~Vasy.
\newblock Local and global boundary rigidity and the geodesic {X}-ray transform
  in the normal gauge.
\newblock {\em arXiv preprint arXiv:1702.03638}, 2017.

\bibitem{stefanov2017local}
P.~Stefanov, G.~Uhlmann, and A.~Vasy.
\newblock Local recovery of the compressional and shear speeds from the
  hyperbolic {DN} map.
\newblock {\em Inverse Problems}, 34(1):014003, 2017.

\bibitem{sun1997inverse}
Z.~Sun and G.~Uhlmann.
\newblock Inverse problems in quasilinear anisotropic media.
\newblock {\em American Journal of Mathematics}, 119(4):771--797, 1997.

\bibitem{uhlmann2016inverse}
G.~Uhlmann and A.~Vasy.
\newblock The inverse problem for the local geodesic ray transform.
\newblock {\em Inventiones mathematicae}, 205(1):83--120, 2016.

\bibitem{uhlmann2018determination}
G.~Uhlmann and Y.~Wang.
\newblock Determination of space-time structures from gravitational
  perturbations.
\newblock {\em Communications on Pure and Applied Mathematics},
  73(6):1315--1367, 2020.

\bibitem{wang2019inverse}
Y.~Wang and T.~Zhou.
\newblock Inverse problems for quadratic derivative nonlinear wave equations.
\newblock {\em Communications in Partial Differential Equations}, pages 1--19,
  2019.

\bibitem{wiechert1907erdbebenwellen}
E.~Wiechert and K.~Zoeppritz.
\newblock {\"U}ber erdbebenwellen.
\newblock {\em Nachrichten von der Gesellschaft der Wissenschaften zu
  G{\"o}ttingen, Mathematisch-Physikalische Klasse}, 1907:415--549, 1907.

\end{thebibliography}

\end{document}